\documentclass{amsart}
\usepackage{tikz}
\usetikzlibrary{arrows.meta}
\usetikzlibrary{decorations.markings}
\usetikzlibrary{calc}

\usepackage{amsfonts}
\usepackage{fullpage}
\usepackage{enumitem}
\usepackage{amsmath} 
\usepackage{amssymb}
\usepackage{amsthm}
\usepackage{xcolor}
\usepackage{hyperref}
\newtheorem{theorem}{Theorem}[section]
\newtheorem{proposition}[theorem]{Proposition}
\newtheorem{lemma}[theorem]{Lemma}
\newtheorem{corollary}[theorem]{Corollary}

\newtheorem{example}[theorem]{Example}
\newtheorem{definition}[theorem]{Definition}
\newtheorem{remark}[theorem]{Remark}

\def\H{\mathrm{H}}
\def\Z{\mathbb{Z}}
\def\E{\mathcal{E}}
\def\D{\mathcal{D}}
\newcommand{\probname}{Crazy Knight's Tour Problem}
\def\G{\Gamma}
\def\R{\mathcal{R}}
\def\C{\mathcal{C}}
\newtheorem*{KN}{Crazy Knight's Tour Problem}

\begin{document}
\title{On relative simple Heffter spaces}

\author[L. Johnson]{Laura Johnson}
\address{School of Mathematics, University of Bristol, Bristol, BS8 1UG, United Kingdom}
\email{laura.marie.johnson@bristol.ac.uk}
\author[L. Mella]{Lorenzo Mella}
\address{Dip. di Scienze Fisiche, Informatiche, Matematiche, Universit\`a degli Studi di Modena e Reggio Emilia, Via Campi 213/A, I-41125 Modena, Italy}
\email{lorenzo.mella@unipr.it}
\author[A. Pasotti]{Anita Pasotti}
\address{DICATAM - Sez. Matematica, Universit\`a degli Studi di Brescia, Via
Branze 43, I-25123 Brescia, Italy}
\email{anita.pasotti@unibs.it}

\keywords{Heffter system, partial linear space, orthogonal cycle decompositions}
%\subjclass[2020]{}

\begin{abstract}
In this paper, we introduce the concept of a relative Heffter space which simultaneously generalizes those of relative Heffter arrays and Heffter spaces. Given a subgroup $J$ of an abelian group $G$, a  relative Heffter space is a resolvable configuration whose points form a half-set of $G\setminus J$ and whose blocks are all zero-sum in $G$. Here we present two infinite families of relative Heffter spaces satisfying the additional condition of being simple. As a consequence, we get new results on globally simple relative Heffter arrays, on mutually orthogonal cycle decompositions and on biembeddings of cyclic cycle decompositions of the complete multipartite graph into an orientable surface.
\end{abstract}

\maketitle

\section{Introduction}
The concept of a Heffter space has been recently introduced in \cite{BP} as a generalization of the well-known notion of a Heffter array, see \cite{A}. In this paper we consider \emph{relative} Heffter spaces which 
%generalize 
are a natural generalization of relative Heffter arrays
%so called in view of their connection with relative difference families, 
introduced in \cite{CMPP} and also of Heffter spaces. Firstly, we recall some necessary concepts and notation.

%I have slightly changed the grammar in the first sentence. It used to read ``having size $\ell$" and I have changed this to ``of size $\ell$" (this is an example of english using``is" where you would use ``have" in italian but because there is already an ``is" in the sentence, we can drop the verb here). I appreciate that there are a lot of ``ofs" in this sentence so I can try and rephrase the sentence if you think it would sound better.
Given an additive group $G$ of order $2v+t$ and a subgroup $J$ of $G$ of order $t$, a \emph{half-set} of $G\setminus J$ is a size $v$ subset $V$ of $G$ such that $V\cup (-V)=G\setminus J$. When $J$ is the trivial subgroup, that is when $t=1$, one simply says that $V$ is a half-set of $G$.

\begin{definition}
Let $G$ be an abelian group of order $2nk+t$, $J$ be a subgroup  of $G$ of order $t$ and 
 $V$ be a half-set of $G\setminus J$. 
An $(nk,k)_t$ \emph{relative Heffter system} on $V$ is a partition  of $V$ into zero-sum parts, called \emph{blocks}, of size $k$.
\end{definition}
In this paper we focus on the case in which $G$ is a cyclic group and we will speak of a \emph{cyclic} relative Heffter system.
When $t=1$, the subscript is omitted and we acquire a classical $(nk,k)$ Heffter system.
%\textcolor{red}{A $(nk,k)_t$ relative Heffter system on $V$ is \emph{integer} if $V$ is contained in $\left[ -(nk+\lfloor\frac{t}{2}\rfloor), nk+ \lfloor\frac{t}{2}\rfloor  \right]$ and each block is zero-sum in $\mathbb{Z}$.}
\begin{definition}
 Two $(nk,k)_t$ relative Heffter systems $\mathcal{P}$ and $\mathcal{Q}$ on the same half-set are \emph{orthogonal} if every block of $\mathcal{P}$
intersects every block of $\mathcal{Q}$ in at most one element.
\end{definition}
%A set of $r$ mutually orthogonal $(v,k)_t$ relative Heffter systems will be denoted, for short, by $(v,k;r)_t$-MOHS.

\begin{example}\label{ex:Hsystem}
The set $V = \{-1,2,3,-4,5,-6,-7,8,-10,11,12,-13,14,-15,-16,17,-19,20,21,-22\}$ is a half-set of $\mathbb{Z}_{45}\backslash{J}$, where  $J$ is the subgroup of $\mathbb{Z}_{45}$ of order $5$. The following sets, $\mathcal{P}$ and $\mathcal{Q}$, are $(20,4)_5$ relative Heffter systems on $V$. 
%\textcolor{red}{We point out that $\mathcal{P}$ is an integer Heffter system, while $\mathcal{Q}$ does not have this property.}

\begin{center}
    \begin{tabular}{|c|}
        \hline
        $\mathcal{P}$ \\
        \hline
        $\{-1,2,3,-4\}$ \\
        \hline 
        $\{5,-6,-7,8\}$ \\
        \hline 
        $\{-10,11,12,-13\}$ \\
        \hline 
        $\{14,-15,-16,17\}$ \\
        \hline 
        $\{-19,20,21,-22\}$ \\
        \hline
    \end{tabular}
    \quad
    \begin{tabular}{|c|}
        \hline
        $\mathcal{Q}$ \\
        \hline 
        $\{-1,11,14,21\}$ \\
        \hline 
        $\{2,8,12,-22\}$ \\
        \hline 
        $\{3,-7,-13,17\}$ \\
        \hline 
        $\{-4,-6,-16,-19\}$ \\
        \hline 
        $\{5,-10,-15,20\}$ \\
        \hline
    \end{tabular}
    \end{center}

In fact, the sets $\mathcal{P}$ and $\mathcal{Q}$ are mutually orthogonal $(20,4)_{5}$ relative Heffter systems since their blocks intersect in at most one element.
\end{example}

An $(nk,k;r)_t$ \emph{relative Heffter space} is nothing more than a set of $r$ mutually orthogonal $(nk,k)_t$ relative Heffter systems.
%such a set is denoted for short by $(v,k;r)_t$-MOHS.
In order to give a more formal definition we have to recall some concepts from classical design theory, see \cite{BJL}. A \emph{partial linear space} (PLS, for short) is a pair $(V, \mathcal{B})$ where $V$ is a set of points and $\mathcal{B}$ is a set of non-empty subsets
(called \emph{blocks} or \emph{lines}) of $V$ with the property that any two distinct points are contained together in at most one block. 
%Two distinct points are said to be \emph{collinear} if there is a block containing them. 
A PLS where every two distinct points are contained in exactly one block is said to be a \emph{linear space}. The \emph{degree} of a point of a PLS is the number of blocks containing that point. A PLS has \emph{degree} $r$ if all its points have the same degree $r$. A \emph{parallel class} of a PLS is a set of blocks partitioning the point set. A PLS is said to be \emph{resolvable} if there exists a partition of the block set (called a \emph{resolution}) into parallel classes. By a \emph{resolved} PLS we mean a resolvable PLS together with a specific resolution of it.
    We will focus on resolvable PLSs in which all blocks have the same size, these are known as \emph{configurations}; clearly a configuration with $v$ points, constant block size $k$ and degree $r$ has necessarily $b=\frac{vr}{k}$ blocks.

\begin{definition}
Given an abelian group $G$ and a subgroup $J$ of $G$,
a \emph{Heffter space over} $G$ \emph{relative to} $J$ is a resolved partial linear space whose parallel classes are mutually orthogonal relative Heffter
systems on a half-set of $G\setminus J$.
\end{definition}
When $G$ is a cyclic group, we will speak of a \emph{cyclic} relative Heffter space.
If $J$ is the trivial subgroup we find the concept of a Heffter space introduced in \cite{BP}.
%\textcolor{red}{We say that the space is \emph{integer} is each of its Heffter systems is integer.}

Note that the degree of the space is nothing but the number of mutually orthogonal relative Heffter systems of the space. The aim of the research on this topic is to construct (relative) Heffter spaces with largest possible degree.
 In \cite{BP1, BP} the authors construct infinite classes of Heffter spaces with an arbitrary large degree. Other very recent results on Heffter spaces have been obtained in \cite{BP2}.
When the degree of the space is $2$, namely when we have only $2$ orthogonal (relative) Heffter systems, the (relative) Heffter space is in fact a square (relative) Heffter array (see \cite{A,CMPP}), which can be formally defined as follows.

\begin{definition}
Let $w = 2nk + t$ be a positive integer and let $J$ be the subgroup of $\mathbb{Z}_w$ of
order $t$. A $\H_t(n;k)$ \emph{Heffter array over} $\mathbb{Z}_w$ \emph{relative to} $J$ is an array of order $n$ with
elements in $\mathbb{Z}_w$ such that:
\begin{itemize}
    \item[\rm{(a)}] each row and each column has exactly $k$ filled cells;
    \item[\rm{(b)}] the entries form a half-set of $\mathbb{Z}_w\setminus J$;
    \item[\rm{(c)}] every row and every column is zero-sum in $\mathbb{Z}_w$.
\end{itemize}
\end{definition} 
 For instance the Heffter systems $\mathcal{P}$ and $\mathcal{Q}$ in Example \ref{ex:Hsystem} are the parallel classes of a $(20,4;2)_5$ relative Heffter space and can be displayed by the following $\H_5(5;4)$ Heffter array  over $\mathbb{Z}_{45}$. 
 $$\begin{array}{|r|r|r|r|r|r|r|} 
        \hline 
        -1 & 2 & 3 & -4 & \\
        \hline 
         & 8 & -7 & -6 & 5 \\
        \hline 
        14 & & 17 & -16 & -15 \\
        \hline 
        21 & -22 & & -19 & 20 \\
        \hline 
        11 & 12 & -13 & & -10 \\
        \hline 
 \end{array}$$
If $t=1$, the subscript notation is omitted, and we find again the concept of a square Heffter array introduced by Archdeacon in \cite{A}, whose existence has been completely established, in fact it is known that there exists a $\H(n;k)$ if and only if $n\geq k \geq3$, see \cite{ADDY,CDDY,DW}.
On the other hand, when $t>1$ the existence problem is still largely open, partial results can be found in \cite{CMPP,CPP,MT,MP1,MP3}.
%below we recall the main known result for the $t > 1$ case in which the order of the subgroup $J$ equals the block-size $k$ (see \cite{CMPP}) other results have been obtained in \cite{CP,CPP,MP1,MP2,MP3}.
%\begin{theorem}\label{thm:esistenza}
%Let $3\leq k\leq n$ with $k\neq 5$.
%  There exists an integer $\H_k(n;k)$ if and only if one of the following holds:
%  \begin{itemize}
%    \item[\rm{(a)}] $k$ is odd and $n\equiv 0,3\pmod 4$;
%    \item[\rm{(b)}] $k\equiv 2\pmod 4$ and $n$ is even;
%    \item[\rm{(c)}] $k\equiv 0\pmod 4$.
%  \end{itemize}
%Furthermore, there exists an integer $\H_5(n;5)$ if
%$n\equiv 3\pmod 4$ and it does not exist if $n\equiv 1,2\pmod 4$.
%\end{theorem}
%Note that for $k=5$ the authors of \cite{CMPP} only solved the existence problem of integer relative Heffter arrays $\H_5(n;5)$ for $n\equiv 3\pmod 4$, leaving the $n\equiv 0 \pmod 4$ case open. 
We point out that in the definitions proposed in \cite{A,CMPP} the elements of the array belong to a cyclic group, but more generally one can consider a Heffter array  with entries in an abelian group; for variants and generalizations of classical Heffter arrays see \cite{PD}.

As explained in \cite{BP}, a Heffter space is more interesting the closer it is to a linear space, a good parameter to measure this distance is the so-called density of the space. The \textit{density} $\delta$ of a $(nk,k;r)_t$ relative Heffter space is defined as the density of the collinear graph associated to the space and, reasoning as in \cite{BP}, one can find that $\delta = \frac{r(k-1)}{nk-1}$. The space is linear if and only if $\delta= 1$. Here we focus on relative Heffter spaces having high density and which are \emph{simple}.
%In this article, we show that for every $\epsilon >0$ there exist infinitely many cyclic relative Heffter spaces having density larger than $1-\epsilon$.  \textbf{ no cyclic linear Heffter space? Maybe after Thm 3.3 as a remark?
One of the motivations for studying simple relative Heffter spaces is that, starting from the blocks of such a space, one can construct a set of mutually orthogonal cycle decompositions of the complete multipartite graph, as explained in Section \ref{sec:decomposition}.
 We say that a $k$-subset $B$ of an abelian group $G$ is \emph{simple} if there exists an ordering $\{b_0,b_1,\ldots,b_{k-1}\}$ of the elements of $B$ such that the $k$-sequence of its partial sums $(c_0,c_1,\ldots,c_{k-1})$, where $c_i = \sum\limits_{j=0}^i b_j$ for $0 \leq i \leq k-1$, does not have any repeated elements. We say that a relative Heffter space is \emph{simple} if each of its blocks  admits a simple ordering.  
Note that since the blocks of a Heffter system sum to $0$, the Heffter system is simple if and only if we can order the elements of the blocks in such a way that no subsequence sums to $0$. 

In this paper we firstly present some preliminary results which allow us to construct, in Section \ref{sec:main}, two infinite classes of simple relative Heffter spaces (see Theorems \ref{thm:hs1} and \ref{thm:hs2}), one of which always achieves the maximal density. Then, in Section \ref{sec:HA}, we get, as a consequence, two new infinite classes of relative Heffter arrays (see Theorems \ref{thm:array1} and \ref{thm:array2}) satisfying the very strong additional condition of being globally simple.
Finally, in the last section, we present new constructive results regarding sets of mutually orthogonal cyclic cycle decompositions of the complete multipartite graph and  biembeddings of these decompositions into an orientable surface.

\section{Preliminary Results}\label{sec:preliminary}
In this section we record some preliminary results that will be used to identify relative Heffter space constructions in Section \ref{sec:main}. 
Given two integers $a$ and $b$, by $[a,b]$ we denote the set $\{a,a+1,\ldots, b\}$ if $a\leq b$, while $[a,b]$ is empty if $a>b.$
Also, given a subset $S$ of $\mathbb{Z}_v$, by $\sum S$ we denote the sum of all the elements in $S$.
Firstly we prove an existence result on simple zero-sum half-sets in a cyclic group. 

\begin{proposition}\label{JMP Proposition}
        Let $k \geq 3$ be an integer. Then there exists a zero-sum half-set of $\mathbb{Z}_{2k+1}$ admitting a simple ordering. 
\end{proposition}
    
    \begin{proof}
    We divide the proof into cases, depending on the value of $k$ modulo $4$.
    
    \textbf{Case $k \equiv 1 \pmod{4}$.} If $k = 5$, it is immediate to check that $L = \{1,-2,3,4,5\}$ is a zero-sum half-set admitting a simple ordering in $\Z_{11}$.
For  $k \geq 9$ consider the following half-set of $\Z_{2k+1}$:
    \[
    \begin{aligned}
    L = &\{-1,-2,3\} \cup \left\{2i,-(2i+1)\mid i\in\left[2, \frac{k+3}{4}\right]\right\} \cup \\ 
    &\left\{-2i,2i+1 \mid i\in\left[\frac{k+7}{4}, \frac{k-3}{2}\right]\right\} \cup \{-(k-1),-k\}.
    \end{aligned}
    \]
    It can be easily seen that the sum of the elements in $L$ is $0$ modulo $2k+1$: the first bracket sums to zero, while the second and the third bracket add to $-\frac{k-1}{4}$ and $\frac{k-9}{4}$,  respectively, and the last one sums to $-2k+1$. A simple ordering of $L$ is then:
    \[
    \begin{aligned}
        \left(-k,-1,4,-5,6,-7,\dotsc, \frac{k+3}{2}, - \frac{k+5}{2}, -2,-(k-1), k-2, -(k-3), \dotsc, -\frac{k+7}{2},3 \right).  
        \end{aligned}
    \]
    Indeed, its partial sums are:
    \[
    \begin{aligned}
    (-k,-k-1,-k+3,-k-2, -k+4,-k-3,\dotsc, \frac{-3k+7}{4}, - \frac{5k+3}{4}, -\frac{5k+11}{4}, \\
    -\frac{k+3}{4},\frac{3k-11}{4},-\frac{k-1}{4},\dotsc, -3,0),
    \end{aligned}
    \]
    that are all distinct.

    \textbf{Case $k \equiv 2 \pmod{4}$.}  We construct the following half-set of $\Z_{2k+1}$:
     \[
     L = \left\{2i-1, -2i \mid i \in \left[1, \frac{k+2}{4}\right] \right\} \cup \left\{ -(2i-1),2i \mid i \in \left[\frac{k+6}{4}, \frac{k}{2}-1\right]\right\} \cup \{-(k-1),-k\}.
     \]
     It can be easily verified that $L$ has zero sum in $\Z_{2k+1}$; a simple ordering of $L$ is:
     \[
     \left(-k,1,-2,3,-4,\dotsc, \frac{k}{2}, -\frac{k+2}{2}, -(k-1),k-2,-(k-3) ,\dotsc, \frac{k}{2}+3, -\frac{k+4}{2}\right),
     \]
     with partial sums
     \[
     \left(-k,-k+1,-k-1,-k+2,-k-2, \dotsc, \frac{-3k+2}{4}, -\frac{5k+2}{4}, -\frac{k-6}{4}, \frac{3k-2}{4}, -\frac{k-10}{4}, \dotsc, \frac{k}{2}+2, 0\right)
     \]
     that are all distinct.

     \textbf{Case $k \equiv 3 \pmod{4}$.} 
     If $k=3$, then choose $L = \{1,2,-3\}$, where clearly any ordering of its elements is simple. Assume then that $k \geq 7$, and consider the following half-set $L$ of $\Z_{2k+1}$:
     \[
     L = \{1\}\cup \left\{2i,-(2i+1) \mid i \in \left[1,\frac{k+1}{4}\right]\right\} \cup \left\{-2i, 2i+1 \mid i \in\left[\frac{k+5}{4}, \frac{k-1}{2}\right]\right\}.
     \]
     Then, a simple ordering for $L$ is:
     \[
     \left(k,2,-3,4,-5,\dotsc, \frac{k+1}{2}, - \frac{k+3}{2}, -(k-1), k-2, -(k-3), \dotsc,\frac{k+7}{2},  -\frac{k+5}{2},1 \right)
     \]
     whose partial sums are:
     \[
     \left(k,k+2,k-1,k+3,k-2,\dotsc, \frac{5k+5}{4}, \frac{3k-1}{4}, -\frac{k-3}{4},\frac{3k-5}{4},\frac{-k+7}{4},\dotsc, \frac{k+3}{2},-1,0\right).
     \]
     \textbf{Case $k \equiv 0 \pmod{4}$.} Construct the half-set $L$ of $\Z_{2k+1}$ defined as:
     \[
    L = \left\{2i-1, -2i \mid i \in\left[ 1,\frac{k}{4}\right]\right\} \cup \left\{-(2i-1), 2i \mid i \in \left[\frac{k}{4}+1, \frac{k}{2}\right] \right\}.
     \]
     A simple ordering is then the following:
     \[
     \left(k,1,-2,3,-4,\dotsc, -\frac{k}{2}, -(k-1),k-2, -(k-3), \dotsc, \frac{k}{2}+2, -\left(\frac{k}{2}+1\right) \right)
     \]
     since its partial sums are:
     \[
     \left(k, k+1, k-1,k+2,k-2, \dotsc,\frac{3k}{4}, -\frac{k}{4}+1,\frac{3k}{4}-1, -\frac{k}{4}+2,\dotsc, \frac{k}{2}+1, 0  \right).
     \]
     % I have removed an extra term from this partial sum.
\end{proof}
We point out that  $L$ sums to zero in $\mathbb{Z}$ for $k\equiv 0,3 \pmod 4$. When this happens we will say that $L$ is an \emph{integer} half-set.
\begin{example}
    In this example we construct a simple ordering of a zero-sum half-set of $\Z_{2k+1}$, following the proof of Proposition \ref{JMP Proposition}, for $k \in [13,16]$.  Note that for each of the partial sums, we have chosen a representative of each congruence class that is in the range $[-k,k]$.
    \begin{footnotesize}
\begin{center}
    \begin{tabular}{ c l l} 
     $k$ & Simple ordering  & Partial sums \\ \hline
   $13$ &  $( -13,   -1,   4,   -5,    6,   -7,    8,   -9,   -2,  -12,   11,  -10,    3)$ & $( -13,   13,  -10,   12,   -9,   11,   -8,   10,    8,   -4,    7,   -3,    0)$ \\ 
   $14$ & $ ( -14,    1,   -2,    3,   -4,    5,   -6,    7,   -8,  -13,   12,  -11,   10,   -9)$
    &    $( -14,  -13,   14,  -12,   13,  -11,   12,  -10,   11,   -2,   10,   -1,    9,    0)$ \\ 
     $15$ & $( 15,    2,   -3,    4,  -5,    6,   -7,    8,   -9,  -14,   13,  -12,   11,  -10,   1)$& $(15,  -14,   14,  -13,   13,  -12,   12,  -11 ,  11,   -3,   10,   -2,    9,   -1,    0)$ \\ 
     $16$& $( 16,    1,   -2,    3,   -4,    5,   -6,    7,   -8,  -15,   14,  -13,   12,  -11,   10,   -9)$&
    $(16,  -16,   15,  -15,   14,  -14,   13, -13,   12,   -3,   11   -2,   10,   -1,    9,    0)$
    \end{tabular}
    \end{center}
    \end{footnotesize}
    
\end{example}    

Now we present two results about a particular collection of sequences. Firstly we consider the case of an odd-length sequence.  

\begin{lemma}\label{sequence lemma}
    Let $k \geq 3$ be an odd integer and $A = (a_0,\ldots,a_{k-1})$ be defined as follows: 
		$$a_i=\begin{cases}1 & {\rm if \ } i=0,k-1\\
		 -2 & {\rm if\ } i {\rm\ odd\ with\ } 1\leq i\leq k-2 \\
		 2 & {\rm if\ } i {\rm\ even \ with\ } 2\leq i\leq k-3. \end{cases}  $$
Then, for $i\in [0,k-1]$ set $A_i=(\alpha_{i,0},\ldots,\alpha_{i,k-1})$ where $\alpha_{i,j}=j\cdot a_{i+j}$ (all subscripts are considered modulo $k$).
It results:
    \begin{itemize}
        \item [\rm{(a)}]  $\sum A= \sum A_0=0$,
        \item [\rm{(b)}]  $\sum A_i=k$ if $i$ is odd, 
		\item [\rm{(c)}] $\sum A_i=-k$ if $i\geq2$ is even.
        \end{itemize}
\end{lemma}

\begin{proof}
    \begin{itemize}
        \item [(a)] Notice that the sequence $A$ comprises $2$ $1$s,  $\frac{k-1}{2}$ $-2$'s and $\frac{k-3}{2}$ $2$'s: it is therefore immediate that this sequence sums to $0$. Since for any $j$ even with $j \in [2, k-3]$ we have $a_{0,j}+a_{0,j+1}=ja_j+(j+1)a_{j+1}=2j-2(j+1)=-2$, then $\sum_{j=2}^{k-2}\alpha_{0,j}=-2\frac{k-3}{2}=-k+3$. Since $\alpha_{0,0}=0$, $\alpha_{0,1}=-2$,  $\alpha_{0,k-1}=k-1$, the thesis follows.

        \item [(b)] If $i=1$, for $j$ odd with $j\in [1,k-4]$ we have $\alpha_{1,j}+\alpha_{1,j+1}=2j-2(j+1)=-2$, hence $\sum_{j=1}^{k-3}a_{1,j} =-2\frac{k-3}{2}=-k+3$. Since $\alpha_{1,0}=0$, $\alpha_{1,k-2}=k-2$, and $\alpha_{1,k-1}=k-1$ the thesis follows. 
        
 Suppose now $i\geq3$ odd. If $j$ is odd with $j\in[1,k-i-3]$ then $\alpha_{i,j}+\alpha_{i,j+1}=ja_{i+j} + (j+1)a_{i+j+1} = 2j-2(j+1)=-2$ since $i+j$ is an even integer not exceeding $k-3$. Hence $\sum_{j=1}^{k-i-2}\alpha_{i,j}=-2\frac{k-i-2}{2}=-k+i+2$. While if $j$ is odd with $j\in[k-i+1,k-2]$ then $\alpha_{i,j}+\alpha_{i,j+1}=-2j+2(j+1)=2$ since $i+j\equiv 2\ell+1\pmod k$ with $\ell\in {[0,\frac{k-5}{2}]}$. Hence $\sum_{j=k-i+1}^{k-1}\alpha_{i,j}=2\frac{i-1}{2}=i-1$.
Note also that $\alpha_{i,0}=0$, $\alpha_{i,k-i-1}=k-i-1$ and $\alpha_{i,k-i}=k-i$.
Then $\sum A_i=(-k+i+2)+(i-1)+(k-i-1)+(k-i)=k$.

 \item [(c)] Suppose now $i\geq2$ even.  If $j$ is odd with $j\in[1,k-i-4]$ then $\alpha_{i,j}+\alpha_{i,j+1}=-2j+2(j+1)=2$ since $i+j$ is an odd integer not exceeding $k-4$. Hence $\sum_{j=1}^{k-i-3}\alpha_{i,j}=2\frac{k-i-3}{2}=k-i-3$. 
If $j$ is even with $j\in[k-i+1,k-3]$ then $\alpha_{i,j}+\alpha_{i,j+1} = -2j+2(j+1)=2$ since $i+j\equiv 2\ell+1\pmod k$ with $\ell\in {[0,\frac{k-4}{2}]}$. Hence $\sum_{j=k-i+1}^{k-2}\alpha_{i,j}=2\frac{i-2}{2}=i-2$.
Note also that $\alpha_{i,0}=0$, $\alpha_{i,k-i-2}=-2(k-i-2)$, $\alpha_{i,k-i-1}=k-i-1$, $\alpha_{i,k-i}=k-i$ and $\alpha_{i,k-1}=-2(k-1)$.
Then $\sum A_i=(k-i-3)-2(k-i-2)+(k-i-1)+(k-i)+(i-2)-2(k-1)=-k$.
% \item [(d)] Let $k = 2\ell+1$ for some integer $\ell$; the statement immediately follows as:
% \[
% \sum_{\textcolor{red}{i=0}}^{\textcolor{red}{k-1}} \textcolor{red}{i}a_i =  \sum_{\textcolor{red}{m=1}}^{\textcolor{red}{\ell}} (\textcolor{red}{(2m-1)a_{2m-1} + 2m a_{2m}}) = 0. 
 %\] 
        \end{itemize}
\end{proof}

\begin{example}
Take $k=7$ then:
\begin{center}
    \begin{tabular}{l l} 
$A = (1,-2,2,-2,2,-2,1)$ &\quad $A_3  =  (0,2,-4,3,4,-10,12)$\\
$A_0  =  (0,-2,4,-6,8,-10,6)$ &\quad $A_4 =  (0,-2,2,3,-8,10,-12)$\\
$A_1 = (0,2,-4,6,-8,5,6)$ &\quad $ A_5 = (0,1,2,-6,8,-10,12)$\\
$A_2 = (0,-2,4,-6,4,5,-12)$ &\quad  $A_6 = (0,1,-4,6,-8,10,-12)$\\
%A_3 & = & (0,2,-4,3,4,-10,12)\\
%A_4 & = & (0,-2,2,3,8,10,-12)\\
%A_5 & = & (0,1,2,-6,8,-10,12)\\
%A_6 & = & (0,1,-4,6,-8,10,-12)\\
\end{tabular}
\end{center}
It is immediate to check that $\sum A=\sum A_0=0$, $\sum A_1=\sum A_3=\sum A_5=7$ and $\sum A_2=\sum A_4=\sum A_6=-7$.
%and that $\sum_{i=1}^7 (i-1)a_i=-2+4-6+8-10+6=0$.
\end{example}

Now we focus on a class of sequences of even length.

\begin{lemma} \label{lem:seq_k_even}
Let $k\geq 4$ be an even integer and $A = (a_0, \dotsc, a_{k-1})$ be defined as follows:
\begin{itemize}
	\item[\rm{(a)}] if  $k \equiv 0 \pmod{4}$
	\[a_i=\begin{cases}1 & {\rm if \ } i\equiv 0,3 \pmod{4},\\
		 -1 & {\rm if\ } i\equiv 1,2 \pmod{4}, \end{cases}
		 \]
		 \item[\rm{(b)}] if $k \equiv 2 \pmod{4}$
	\[a_i=\begin{cases}
		-2 & {\rm if \ } i=0,\\
		 2 & {\rm if \ } i=1,\\
		1 & {\rm if \ } i \in \{2,4\} {\rm  \ or \ } i\equiv 1,2 \pmod{4} {\rm \  and  \ } i\geq 6,\\
		 -1 &{\rm if \ } i \in \{3,5\} {\rm  \ or \ } i\equiv 0,3 \pmod{4} {\rm \  and  \ } i>6. \\
		 \end{cases}
		 \]
\end{itemize}
Then $\sum A=\sum_{i=0}^{k-1} ia_i= 0$.
%\begin{itemize} 
%\item[\rm{(a)}] $\sum A=0$=;
%\item[\rm{(b)}] $  \sum_{i=1}^k a_i (i-1)  = 0$.
%\end{itemize}
\end{lemma}

\begin{proof}
From the definition of $A$, it is immediate to verify that $\sum A=0$. To check that  $\sum_{i=0}^{k-1} ia_i   = 0$, assume first that $k = 4 \ell$ for some $\ell \geq1$, hence: 
\[
\begin{aligned}
 \sum_{i=0}^{k-1} ia_i &= \sum_{j=0}^{\ell-1} (4j\,a_{4j} + (4j+1)\, a_{4j+1}+ (4j+2)\,a_{4j+2}+(4j+3)\,a_{4j+3}) \\
 & =  \sum_{j=0}^{\ell-1} (4j - (4j+1)- (4j+2)+(4j+3)) = 0.
 \end{aligned}
\]
If $k = 4\ell+2$ for some $\ell \geq 1$, then: 
 \[
\begin{aligned}
 \sum_{i=0}^{k-1} ia_i &= (1a_1+2a_2+\dotsc+5a_5)+ \sum_{j=1}^{\ell-1} ((4j+2)a_{4j+2}+(4j+3)a_{4j+3}+(4j+4)a_{4j+4}+(4j+5)a_{4j+5}) \\
 &= (2+2-3+4-5)+ \sum_{j=1}^{\ell-1} ((4j+2)-(4j+3)-(4j+4)+(4j+5))=0. \\
 \end{aligned}
\]
\end{proof}

We conclude this section by covering some group theoretical results. By ${\rm U}(\mathbb{Z}_w)$ we denote the group of units of the cyclic group $\mathbb{Z}_w$
of order $w$. Also, given $s \in \mathbb{Z}_w$, by $\langle{s}\rangle$ we mean the additive subgroup of $\mathbb{Z}_w$ generated by $s$.

\begin{lemma}\label{inverse cosets lemma}
    Let $s,a \in \mathbb{Z}_w$ with $a \not\in \langle{s}\rangle$. Then the additive inverse of every element of the additive coset $a + \langle{s}\rangle$ is contained within the additive coset $s-a+ \langle{s}\rangle$ and vice versa. Moreover $(a + \langle{s}\rangle) \cap (s-a + \langle{s}\rangle) = \emptyset$. 
\end{lemma}

\begin{lemma}\label{expressing coset in terms of smaller coset lem}
    Let $w = kd(2k+1)$, then every element of the subgroup $\langle{2k+1}\rangle$ of $\mathbb{Z}_w$ can be expressed as a unique element of the form $(id + j)(2k+1)$, where $ i \in [0, k-1]$ and $j \in[0,d-1]$. 
\end{lemma}

\begin{proof}
    Observe that the additive subgroup $\langle{2k+1}\rangle$  of the group $\mathbb{Z}_w$ can be partitioned into $d$ cosets of the smaller additive subgroup $\langle{d(2k+1)}\rangle$ of $\mathbb{Z}_w$, which has cardinality $k$. More specifically, for each $j \in[0, d-1]$, the coset $j(2k+1) + \langle{d(2k+1)}\rangle$ of $\langle{d(2k+1)}\rangle$ is a subset   $\langle{2k+1}\rangle$. Notice that every element of the additive subgroup $\langle{d(2k+1)}\rangle$ may be written as $id(2k+1)$ for some $ i \in[0, k-1]$, therefore every element of the subgroup $\langle{2k+1}\rangle$ can be written in the form $(id+j)(2k+1)$, where $i \in [0, k-1]$ and $j \in [0,d-1]$.
\end{proof}

\section{Constructions of simple relative Heffter spaces}\label{sec:main}
Now we use the results of Section \ref{sec:preliminary} to construct two infinite classes of simple relative Heffter spaces. In both cases, the points of the space form a half-set of $\mathbb{Z}_{n(2k+1)}\setminus \langle{2k+1}\rangle$ for suitable choices of $n$ and $k$. We then show that when $n=k$ is a prime, the constructed Heffter spaces are as dense as possible.
%As a consequence, we will get new results on relative Heffter arrays, see Section \ref{sec:HA}, and on cyclic cycle decompositions, see Section \ref{sec:decomposition}.
%\textcolor{red}{+app to biembeddings}

As usual by $\Phi(k)$ we will denote Euler's totient function of a positive integer $k$.

\begin{proposition}\label{Heffter system 1}
Let $n$ be an odd integer and $k$ be a divisor of $n$.
Then there exist at least $\Phi(k)+1$ simple cyclic $(nk,k)_n$ relative Heffter systems.
%in $\mathbb{Z}_{n(2k+1)}$ relative to $\langle{2k+1}\rangle$.
\end{proposition}

\begin{proof}
Set $n=kd$ and $w = n(2k+1)$. Let $A = (a_0,\ldots,a_{k-1})$ be the zero-sum sequence in $\mathbb{Z}_{2k+1}$ constructed in Lemma \ref{sequence lemma}. 
By Proposition \ref{JMP Proposition} there exists a zero-sum half-set, say $L'$, of $\mathbb{Z}_{2k+1}$ with a simple 
ordering, say $\omega'=(\ell'_0,\ell'_1,\ldots,\ell'_{k-1})$.
Note that the sum of $L'$ is of the form $\alpha(2k+1)$, where $\alpha\in[-1,1]$.  
Set $\omega=(\ell_0=\ell'_0-\alpha(2k+1),\ell_1=\ell'_1,\ldots,\ell_{k-1}=\ell'_{k-1})$,
then $\omega$ is a simple ordering of the half-set $L=\{\ell_0,\ell_1,\ldots,\ell_{k-1}\}$ of $\mathbb{Z}_{2k+1}$ which clearly is integer.
%Also let  $\omega = (\ell_1,\ldots,\ell_k)$ be a simple ordering of an integer half-set $L$ of $\mathbb{Z}_{2k+1}$.
Now set $\mathcal{P}_0 = \{B_{i,j,0} \mid i\in[0,k-1], j\in[0, d-1]\}$, where
    \begin{gather*}
    B_{i,j,0} := \{a_0(id+j)(2k+1)+\ell_0,a_1(id+j)(2k+1)+\ell_1,a_2(id+j)(2k+1)+\ell_2,\ldots,\\
    a_{k-3}(id+j)(2k+1)+\ell_{k-3},a_{k-2}(id+j)(2k+1)+\ell_{k-2},a_{k-1}(id+j)(2k+1)+\ell_{k-1}\}.
    \end{gather*} 
    We will show that $\mathcal{P}_0 $ is a simple  Heffter system in $\mathbb{Z}_{w}$ relative to $\langle{2k+1}\rangle$.
    We start by proving that each block $B_{i,j,0}$ sums to zero in $\mathbb{Z}_w$. Note that
$$\sum B_{i,j,0}=\sum_{m=0}^{k-1} \left(a_m(id+j)(2k+1)+\ell_{m}\right)=(id+j)(2k+1) \sum A +\sum L=0.$$

   Now we prove that the blocks of $\mathcal{P}_0$ partition a half-set of $\mathbb{Z}_w\backslash\langle{2k+1}\rangle$. Notice that we can partition the set $\mathbb{Z}_w\backslash\langle{2k+1}\rangle$ into $2k$ non-trivial cosets of the additive subgroup $\langle{2k+1}\rangle$. It follows from Lemma \ref{inverse cosets lemma} that for each $a\in[1,k]$ the additive inverses of the elements contained within the coset $a + \langle{2k + 1}\rangle$ are all contained within the coset $2k+1-a + \langle{2k+1}\rangle$, therefore, if we can prove that $\mathcal{P}_0$ either contains a copy of the coset $a + \langle{2k + 1}\rangle$ or $2k+1-a + \langle{2k + 1}\rangle$ for all $a\in[1,k]$, then it follows that $\mathcal{P}_0$ is a partition of a half-set of $\mathbb{Z}_v\backslash\langle{2k+1}\rangle$. Since each element of L is either in the coset $a+\langle{2k+1}\rangle$ or $2k+1-a+\langle{2k+1}\rangle$ for each $a \in [1,k]$, and by Lemma \ref{expressing coset in terms of smaller coset lem} each unique element of the coset $\langle{2k+1}\rangle$ may be expressed $(id+j)(2k+1)$ for some $i \in [0,k-1]$ and $j \in [0,d-1]$, it is immediate that $\mathcal{P}_0$ is a relative Heffter system of $\mathbb{Z}_v\backslash\langle{2k+1}\rangle$.

It remains to demonstrate that the Heffter system $\mathcal{P}_0$ is simple. To see this, let $\omega' = (a_0(id+j)(2k+1) + \ell_0,\ldots,a_{k-1}(id+j)(2k+1)+\ell_{k-1})$ be an ordering of the elements of an arbitrary block $B_{i,j,0}$ of $\mathcal{P}_0$, where $i \in [0,k-1]$ and $j\in[0,d-1]$. Observe that since the ordering $\omega = (\ell_0,\ldots,\ell_{k-1})$ of $L$ is simple in the group $\mathbb{Z}_{2k+1}$, none of the partial sums of $\omega$ sum to $0$ modulo $2k+1$. It then follows that, since all multiples of $2k+1$ reduce to $0$ modulo $2k+1$, the partial sums of  $\omega' = (a_0(id+j)(2k+1) + \ell_0,\ldots,a_{k-1}(id+j)(2k+1)+\ell_{k-1})$ will also not sum to $0$ modulo $2k+1$ and hence the partial sums of $\omega'$ will not sum to $0$ modulo $kd(2k+1)$. It therefore follows that $\mathcal{P}_0$ is a simple Heffter system.

For any $s \in {\rm U}(\mathbb{Z}_{2k+1})$ define now $\mathcal{P}_s = \{B_{i,j,s} \mid i\in[0,k-1],j\in[0,d-1] \}$ where
    \begin{gather*}
        B_{i,j,s} = \{a_{i}j(2k+1) + \ell_{i},a_{1+i}(ds+j)(2k+1) + \ell_{1+i},
        a_{2+i}(2ds+j)(2k+1) + \ell_{2+i},\\\ldots,a_{k-2+i}((k-2)ds+j)(2k+1)+\ell_{k-2+i},a_{k-1+i}((k-1)ds+j)(2k+1) +\ell_{k-1+i}\} 
    \end{gather*}
and all subscripts are considered modulo $k$.

We prove that each $\mathcal{P}_s$ is a simple Heffter system in $\mathbb{Z}_{w}$ relative to $\langle{2k+1}\rangle$. Notice that the elements of each block $B_{i,j,s} \in \mathcal{P}_s$ sum to
    $$\sum_{m=0}^{k-1} [a_{m+i}(mds+j)(2k+1)+\ell_{m+i}]= (2k+1)\sum_{m=0}^{k-1} a_{m+i}(mds+j)+ \sum_{m=0}^{k-1} \ell_{m+i}=$$  
 $$ (2k+1)ds\sum_{m=0}^{k-1} ma_{m+i}+j(2k+1)\sum_{m=0}^{k-1} a_{m+i}+\sum L.$$
Since $\sum_{m=0}^{k-1} a_{m+i}=\sum A=0$ and $\sum L=0$ we have
$$\sum  B_{i,j,s}=(2k+1)ds\sum_{m=0}^{k-1} ma_{m+i}.$$
Notice that $\sum_{m=0}^{k-1} ma_{m+i}=\sum A_i$, where $A_i$ is the sequence defined in Lemma \ref{sequence lemma}, is $0, k$ or $-k$.
In each case we get $\sum  B_{i,j,s}\equiv 0 \pmod{kd(2k+1)}$, that is $B_{i,j,s}$ sums to zero in $\mathbb{Z}_w$.
Now it just remains to prove that for any $s \in {\rm U}(\mathbb{Z}_{2k+1})$, the elements of the blocks of $\mathcal{P}_s$ form a half-set of $\mathbb{Z}_w\backslash\langle{2k+1}\rangle$. 
%As in the proof of Proposition \ref{Heffter system 1}, we will do this by proving that for each $\ell_m\in L$ the elements of the coset $\ell_m + \langle{2k+1}\rangle$ are contained in $\mathcal{P}_s$. Observe that in the proof of Proposition \ref{Heffter system 1} it was demonstrated that the elements of the coset $\ell_m + \langle{2k+1}\rangle$ can be expressed by $\ell_m + a_m(id+j)(2k+1)$. Notice that since $s \in U(\mathbb{Z}_k)$, multiplication by $s$ will simply permute the elements of the coset $\ell_m + \langle{2k+1}\rangle$, hence we may also express the elements of this coset by $\ell_m + a_m(ids+j)(2k+1)$, where each unique element is defined by a unique $0 \leq i \leq k-1$ and $0 \leq j \leq d-1$. 
Previously it was demonstrated that each element of the coset $\ell_m + \langle{2k+1}\rangle$ may be uniquely expressed by $\ell_m + a_m(id+j)(2k+1)$, where $a_m \in A$ is fixed, $i \in[0, k-1]$ and $j\in[0,d-1]$. Since $s \in {\rm U}(\mathbb{Z}_{2k+1})$, therefore it follows that each element of the coset $\ell_m + \langle{2k+1}\rangle$ can also be expressed by $\ell_m + a_m(ids+j)(2k+1)$. This means that every block $B_{i,j,s} \in\mathcal{P}_s$ contains precisely one element of each coset $\ell_m + a_m(ids+j)(2k+1)$, where $\ell_m$ is a member of the half-set $L$ of $\mathbb{Z}_{2k+1}$, hence the elements of the blocks of $\mathcal{P}_s$ form a half-set of $\mathbb{Z}_w\backslash\langle{2k+1}\rangle$. 
    Finally, reasoning as done for $\mathcal{P}_0$, one can prove that for any $s$ in ${\rm U}(\mathbb{Z}_{2k+1})$
the Heffter system $\mathcal{P}_s$ is simple.

Since $\Phi(k)=|U(\mathbb{Z}_{2k+1})|$, we have the thesis.
\end{proof}

\begin{example}\label{ex:n15k5}
    Let $n=15$ and $k=5$, then $d=3$, $2k+1 = 11$ and $w = 165$. For $k=5$ we have $A=(1,-2,2,-2,1)$ and  $\omega=(-10,-2,3,4,5)$.  Following the proof of Proposition \ref{Heffter system 1}, we get the five simple cyclic $(75,5)_{15}$ relative Heffter systems $\mathcal{P}_0,\ldots,\mathcal{P}_4$ listed below:
    \begin{center}
        \begin{tabular}{|c|c|c|c|c|}
            \hline 
             & $\mathcal{P}_0$ & $\mathcal{P}_1$ & $\mathcal{P}_2$  \\
           \hline
           $B_{0,0,s}$ & $\{-10,-2,3,4,5\}$ & $\{-10,-68,-30,-29,-28\}$ & $\{-10,31,-63,-62,-61\}$ \\
           \hline
            $B_{1,0,s}$ & $\{23,-68,69,-62,38\}$ & $\{-2,69,37,-61,-43\}$ & $\{-2,-30,70,38,-76\}$ \\
            \hline
            $B_{2,0,s}$ & $\{56,31,-30,37,71\}$ & $\{3,-62,71,-76,64\}$ & $\{3,37,-28,23,-35\}$ \\
            \hline
           $B_{3,0,s}$ & $\{-76,-35,36,-29,-61\}$ & $\{4,38,56,-35,-63\}$ & $\{4,71,-43,-68,36\}$ \\
            \hline
            $B_{4,0,s}$ & $\{-43,64,-63,70,-28\}$ & $\{5,23,31,36,70\}$ & $\{5,56,64,69,-29\}$ \\
           \hline 
           $B_{0,1,s}$ & $\{1,-24,25,-18,16\}$ & $\{1,75,-8,-51,-17\}$ & $\{1,9,-41,81,-50\}$ \\
           \hline
           $B_{1,1,s}$ & $\{34,75,-74,81,49\}$ & $\{-24,-74,15,-50,-32\}$ & $\{-24,-8,48,49,-65\}$ \\
           \hline
            $B_{2,1,s}$ & $\{67,9,-8,15,82\}$ & $\{25,81,82,-65,42\}$ & $\{25,15,-17,34,-57\}$ \\
           \hline
            $B_{3,1,s}$ & $\{-65,-57,58,-51,-50\}$ & $\{-18,49,67,-57,-41\}$ & $\{-18,82,-32,75,58\}$ \\
           \hline
            $B_{4,1,s}$ & $\{-32,42,-41,48,-17\}$ & $\{16,34,9,58,48\}$ & $\{16,67,42,-74,-51\}$ \\
           \hline
            $B_{0,2,s}$ & $\{12,-46,47,-40,27\}$ & $\{12,53,14,-73,-6\}$ & $\{12,-13,-19,59,-39\}$ \\
           \hline
            $B_{1,2,s}$ & $\{45,53,-52,59,60\}$ & $\{-46,-52,-7,-39,-21\}$ & $\{-46,14,26,60,-54\}$ \\
           \hline
            $B_{2,2,s}$ & $\{78,-13,14,-7,-72\}$ & $\{47,59,-72,-54,20\}$ & $\{47,-7,-6,45,-79\}$ \\
           \hline
            $B_{3,2,s}$ & $\{-54,-79,80,-73,-39\}$ & $\{-40,60,78,-79,-19\}$ & $\{-40,-72,-21,53,80\}$ \\
           \hline 
            $B_{4,2,s}$ & $\{-21,20,-19,26,-6\}$ & $\{27,45,-13,80,26\}$ & $\{27,78,20,-52,-73\}$ \\
            \hline
        \end{tabular}
    \end{center}
   \begin{center}
        \begin{tabular}{|c|c|c|}
           \hline
             & $\mathcal{P}_3$ & $\mathcal{P}_4$ \\
           \hline
            $B_{0,0,s}$ & $\{-10,-35,69,70,71\}$ & $\{-10,64,36,37,38\}$ \\
            \hline 
            $B_{1,0,s}$ & $\{-2,36,-62,-28,56\}$ & $\{-2,-63,-29,71,23\}$ \\ 
           \hline 
            $B_{2,0,s}$ & $\{3,-29,38,-43,31\}$ & $\{3,70,-61,56,-68\}$ \\
            \hline
            $B_{3,0,s}$ & $\{4,-61,23,64,-30\}$ & $\{4,-28,-76,31,69\}$ \\
           \hline 
            $B_{4,0,s}$ & $\{5,-76,-68,-63,37\}$ & $\{5,-43,-35,-30,-62\}$ \\
            \hline 
            $B_{0,1,s}$ & $\{1,-57,-74,48,82\}$ & $\{1,42,58,15,49\}$ \\
            \hline
            $B_{1,1,s}$ & $\{-24,58,81,-17,67\}$ & $\{-24,-41,-51,82,34\}$ \\
            \hline
            $B_{2,1,s}$ & $\{25,-51,49,-32,9\}$ & $\{25,48,-50,67,75\}$ \\
            \hline
            $B_{3,1,s}$ & $\{-18,-50,34,42,-8\}$ & $\{-18,-17,-65,9,-74\}$ \\
            \hline
            $B_{4,1,s}$ & $\{16,-65,75,-41,15\}$ & $\{16,-32,-57,-8,81\}$ \\
            \hline
            $B_{0,2,s}$ & $\{12,-79,-52,26,-72\}$ & $\{12,20,80,-7,60\}$ \\
            \hline
            $B_{1,2,s}$ & $\{-46,80,59,-6,78\}$ & $\{-46,-19,-73,-72,45\}$ \\
            \hline
            $B_{2,2,s}$ & $\{47,-73,60,-21,-13\}$ & $\{47,26,-39,78,53\}$ \\
            \hline
            $B_{3,2,s}$ & $\{-40,-39,45,20,14\}$ & $\{-40,-6,-54,-13,-52\}$ \\
            \hline 
            $B_{4,2,s}$ & $\{27,-54,53,-19,-7\}$ & $\{27,-21,-79,14,59\}$ \\
            \hline
        \end{tabular}
    \end{center}

One can directly check that each $B_{i,j,s}$ sums to zero modulo $165$, that the elements of the blocks of $\mathcal{P}_{m}$, for $m \in[0,4]$, form a half-set of $\mathbb{Z}_{165}\setminus \langle11\rangle$. Also, since $k=5$, it is trivial that the partial sums of each block are pairwise distinct. Hence each $\mathcal{P}_i$ is a simple Heffter system in $\mathbb{Z}_{165}$ relative to $\langle11\rangle$.
\end{example}
We now show that the simple relative Heffter systems constructed in Proposition \ref{Heffter system 1} are mutually orthogonal, namely that can be viewed as the parallel classes of a Heffter space.

\begin{theorem}\label{thm:hs1}
Let $n$ be an odd integer and let $k$ be a divisor of $n$. Then there exists a simple cyclic  $(nk,k;\Phi(k)+1)_n$ relative Heffter space.

 %   Let $v = n(2k+1)$ with $n$ odd and $k$ a divisor of $n$. Then there exists a set of  simple $(nk,k;u)$-MOHS over $\mathbb{Z}_v\backslash\langle{2k+1}\rangle$, where $u = \Phi(k)+1$. 
\end{theorem}
\begin{proof}
Set $n=kd$, $w = n(2k+1)$ and $r=\Phi(k)+1$, and let $I = \{0\}\cup {\rm U}(\Z_{2k+1})$.  In Proposition \ref{Heffter system 1} we constructed $r$ simple Heffter systems in $\mathbb{Z}_w$ relative to $\langle2k+1\rangle$. We will denote these Heffter systems by $\mathcal{P}_i$ for each $i \in I$,
%and $\mathcal{P}_s $ for each $ s \in {\rm U}(\Z_{2k+1})$, 
where each of the $\mathcal{P}_i$'s exactly corresponds to the Heffter system denoted in the same way in Proposition \ref{Heffter system 1}.
%where $\mathcal{P}_0$ exactly corresponds to the Heffter system denoted in the same way in Proposition \ref{Heffter system 1} and each of $\mathcal{P}_1,\ldots,\mathcal{P}_{r-1}$ corresponds to one of the remaining $r-1$ Heffter systems outlined in this proposition.
If we fix $L$ to be the same half-set of $\mathbb{Z}_{2k+1}$, then for each $\ell_m \in L$ %the parallel classes $\mathcal{P}_0,\ldots,\mathcal{P}_{r-1}$ 
every Heffter system $\mathcal{P}_i$
contains elements of each coset $\ell_m + \langle{2k+1}\rangle$ of the subgroup $\langle{2k+1}\rangle$ of $\mathbb{Z}_w$ i.e. %for each $i \in I$ every $\mathcal{P}_i$ contains elements of the same half-set of $\mathbb{Z}_w\backslash\langle{2k+1}\rangle$.

To prove that these Heffter systems are mutually orthogonal, we simply need to prove that for any $t_1, t_2 \in I,$ $t_1\neq t_2$ any block of  $\mathcal{P}_{t_1}$ intersects any block of $\mathcal{P}_{t_2}$ in at most one element.
Suppose firstly $t_1=0$ and $t_2=s \in {\rm U}(\mathbb{Z}_{2k+1})$. Note that given two blocks $B_{i_1,j_1,0}\in \mathcal{P}_0$ and $B_{i_2,j_2,s}\in \mathcal{P}_{s}$ with $j_1\neq j_2$ then the elements of the two blocks are contained in different cosets of the subgroup $\langle d(2k+1)\rangle$
which implies $B_{i_1,j_1,0} \cap B_{i_2,j_2,s}=\emptyset$. Hence we can take $B_{h,j,0}\in \mathcal{P}_0$ and $B_{i,j,s}\in \mathcal{P}_{s}$, that is:
 \begin{gather*}
    B_{h,j,0} := \{a_0(hd+j)(2k+1)+\ell_0,a_1(hd+j)(2k+1)+\ell_1,a_2(hd+j)(2k+1)+\ell_2,\ldots,\\ a_{k-3}(hd+j)(2k+1)+\ell_{k-3},{a_{k-2}}(hd+j)(2k+1)+\ell_{k-2},a_{k-1}(hd+j)(2k+1)+\ell_{k-1}\},
    \end{gather*}
and 
    \begin{gather*}
        B_{i,j,s} = \{a_{i}j(2k+1) + \ell_{i},a_{1+i}(ds+j)(2k+1) + \ell_{1+i},a_{2+i}(2ds+j)(2k+1)+ \ell_{2+i}, \\ \phantom{B_{(i,j)} +} \ldots,a_{k-2+i}((i-2)ds+j)(2k+1)+\ell_{k-2+i},a_{k-1+i}((k-1)ds+j)(2k+1)+ \ell_{k-1+i}\}.
    \end{gather*}
 By way of contradiction suppose that there exists an $s \in {\rm U}(\mathbb{Z}_{2k+1})$ such that a block
 $B_{i,j,s}\in \mathcal{P}_{s}$ intersects with a block of $\mathcal{P}_0$ in two distinct elements $x$ and $y$. 
Let then $m_1\in [0, k-1]$  be the index such that $x \equiv \ell_{m_1} \pmod{2k+1}$. Since $x \in B_{h,j,0} \cap B_{i,j,s}$, the following equation is satisfied:
\[
a_{m_1}(hd+j)(2k+1) + \ell_{m_1} = a_{m_1}((m_1-i)ds+j)(2k+1) + \ell_{m_1}.
\]
This implies:
\begin{equation} \label{eq:h1}
h = (m_1-i)s.
\end{equation}
An analogous argument can be carried out for $y \equiv \ell_{m_2} \pmod{2k+1}$ for some $m_2\in[0, k-1]$, $m_2 \neq m_1$, yielding that
\begin{equation} \label{eq:h2}
h = (m_2-i)s.
\end{equation}
Since $m_1,m_2\in[0, k-1]$, $m_2 \neq m_1$, and $s \in {\rm U}(\mathbb{Z}_{2k+1})$, it is clear that (\ref{eq:h1}) and (\ref{eq:h2}) cannot be satisfied at the same time, so we reach a contradiction. \\
    \\
  Similarly, assume by contradiction that for two distinct $s_1,s_2 \in {\rm U}(\mathbb{Z}_{2k+1})$
there exist two blocks $B_{i_1,j_1,s_1} \in \mathcal{P}_{s_1}$  and $B_{i_2,j_2,s_2} \in \mathcal{P}_{s_2}$ intersecting in two distinct elements $x$ and $y$. Similarly to the previous case, we can immediately deduce that $j_1=j_2=j$, and if $x \equiv \ell_1 \pmod{2k+1}$ for some $m_1\in[0, k-1]$, then
$x \in B_{i_1,j,s_1}\cap B_{i_2,j,s_2}$ implies
\[
a_{m_1}((m_1-i_1)ds_1+j)(2k+1) +\ell_{m_1}= a_{m_1}((m_1-i_2)ds_2+j)(2k+1)+\ell_{m_1}.
\]
 After some computations, we obtain the following:
\begin{equation} \label{eq:m1}
   (m_1 - i_1)s_1 =(m_1 - i_2)s_2. 
\end{equation}

Similarly, if $y \equiv \ell_{m_2} \pmod{2k+1}$ for some $m_2\in[0, k-1]$, $m_2 \neq m_1$, then:
\[
 a_{m_2}((m_2-i_1)ds_1+j)(2k+1) + \ell_{m_2} = a_{m_2}((m_2-i_2))ds_2+j)(2k+1)+ \ell_{m_2}.
\]
We then obtain:
\begin{equation} \label{eq:m2}
    (m_2 - i_1)s_1 =(m_2 -i_2)s_2.
\end{equation}
 Notice that if we subtract Equation (\ref{eq:m2}) from Equation (\ref{eq:m1}), we obtain the following:
    \begin{equation}\label{eq:finale}
    (m_1-m_2)s_1 = (m_1-m_2)s_2.
    \end{equation}
Since every unit $s \in {\rm U}(\mathbb{Z}_{2k+1})$ maps each group element $z \in \mathbb{Z}_{2k+1}$ to a unique group element $zs \in \mathbb{Z}_{2k+1}$, it follows that Equation (\ref{eq:finale}) can only hold if $s_1 = s_2$. This is a contradiction.

    Moreover, note that since each of the Heffter systems is simple, the relative Heffter space is simple.
\end{proof}

\begin{remark}
When $n=k$, the above construction produces several Heffter spaces with density $\delta \geq 0.7$. In particular, when $n=k$ is prime we obtain a Heffter space with $\delta = \frac{n}{n+1}$, which is the densest possible Heffter space in a cyclic group (see also Remark \ref{rem:densest_HS}). When $n=k=p^m$ is a prime power, then $\delta = \frac{p^{m-1}(p-1)}{p^m+1}$, which tends to a value of $\delta \geq 0.8$ as $m \to \infty$. Finally, when $n = k= pq$ is semiprimitive, we obtain a  Heffter space with $\delta = \frac{(p-1)(q-1)}{pq + 1}$. If $p \leq q$, then we get a Heffter space with $\delta \geq 0.7$ for $p \geq 5$ and $q \geq 11$ or if $p,q \geq 7$. 
\end{remark}

\begin{example}

One can easily check that the blocks of the Heffter systems constructed in Example 
\ref{ex:n15k5}  intersect each other in at most one point, meaning that they form a set of $5$ mutually orthogonal Heffter systems.
In other words the blocks of Example \ref{ex:n15k5} form a simple cyclic $(75,5;5)_{15}$ relative Heffter space whose five parallel classes are $\mathcal{P}_0,\ldots,\mathcal{P}_4$.
\end{example}

%\textcolor{red}{I think we can remove the following corollary.}
%The following result is then immediate from the above. 

%\begin{corollary} \label{cor:np_p_p}
%Let $n$ be an odd integer and $p$ be a prime divisor on $n$. Then there exists a simple cyclic $(np,p;p)_{n}$ relative Heffter space.
 %   Let $v = n(2k+1)$, with $n = kp$, and $k$ a prime divisors of $n$, then there exists a set of simple $(nk,k;k)$-MOHS. In particular, when $n=k$ (i.e. $n=k$), we have a set of simple $(n^2,n;n)$-MOHS over $\mathbb{Z}_{n(2n+1)}\backslash\langle{2n+1}\rangle$.
%\end{corollary}

\begin{example}
Take $n=k=5$, then $2k+1 = 11$, $v = 55$, and $ \Phi(k)+1 = 5$. Hence there exists a simple cyclic $(25,5;5)_{5}$ relative Heffter space whose five parallel classes $\mathcal{P}_0,\ldots, \mathcal{P}_4$ are listed below. 
    \begin{center}
        \begin{tabular}{|c|c|c|c|c|}
            \hline 
             & $\mathcal{P}_0$ & $\mathcal{P}_1$ & $\mathcal{P}_2$ \\
            \hline
            $B_{0,0,s}$ & $\{-10,-2,3,4,5\}$ & $\{-10,-24,-8,-7,-6\}$ & $\{-10,9,-19,-18,-17\}$ \\
            \hline
            $B_{1,0,s}$ & $\{1,-24,25,-18,16\}$ & $\{1,9,14,26,5\}$ & $\{1,-13,3,15,-6\}$ \\
            \hline
            $B_{2,0,s}$ & $\{12,9,-8,15,27\}$ & $\{12,-13,-19,4,16\}$ & $\{12,20,25,-7,5\}$ \\
            \hline
            $B_{3,0,s}$ & $\{23,-13,14,-7,-17\}$ & $\{23,20,3,-18,27\}$ & $\{23,-2,-8,26,16\}$ \\
            \hline
            $B_{4,0,s}$ & $\{-21,20,-19,26,-6\}$ & $\{-21,-2,25,15,-17\}$ & $\{-21,-24,14,4,27\}$ \\
            \hline 
        \end{tabular}
    \end{center}
    \begin{center}
        \begin{tabular}{|c|c|c|}
            \hline
             & $\mathcal{P}_3$ & $\mathcal{P}_4$ \\
            \hline
            $B_{0,0,s}$ & $\{-10,-13,25,26,27\}$ & $\{-10,20,14,15,16\}$  \\
            \hline 
            $B_{1,0,s}$ & $\{1,20,-8,4,-17\}$ & $\{1,-2,-19,-7,27\}$ \\ 
            \hline 
            $B_{2,0,s}$ & $\{12,-2,14,-18,-6\}$ & $\{12,-24,3,26,-17\}$ \\
            \hline
            $B_{3,0,s}$ & $\{23,-24,-19,15,5\}$ & $\{23,9,25,4,-6\}$ \\
            \hline 
            $B_{4,0,s}$ & $\{-21,9,3,-7,16\}$ & $\{-21,-13,-8,-18,5\}$ \\
            \hline
        \end{tabular}
    \end{center}
\end{example}

In the next theorem we present another direct construction of an infinite family of simple cyclic relative Heffter spaces.

\begin{theorem}\label{thm:hs2}
Let $p\geq 3$ be a prime and let $k \in[3,p]$. Then there exists a simple cyclic $(pk,k;p)_{p}$ relative Heffter space. 
\end{theorem}

\begin{proof}
Let $p$ and $k$ be as in the statement.  Let $A = (a_0,\ldots,a_{k-1})$ be the  sequence of $\mathbb{Z}_{2k+1}$ constructed in Lemma \ref{sequence lemma} (respectively in Lemma \ref{lem:seq_k_even}) if $k$ is odd (respectively if $k$ is even). 
As done in the proof of Proposition \ref{Heffter system 1}, we can construct an integer half-set $L = \{\ell_0,\ell_1,\ldots,\ell_{k-1}\}$ of $\mathbb{Z}_{2k+1}$ having a simple ordering.

%By Proposition \ref{JMP Proposition} there exists a zero-sum half-set, say $L'$, of $\mathbb{Z}_{2k+1}$ with a simple 
%ordering, say $\omega'=(\ell'_1,\ell'_2,\ldots,\ell'_k)$.
%Note that the sum of $L'$ is of the form $\alpha(2k+1)$, where $\alpha\in[-1,1]$.  
%Set $\omega=(\ell_1=\ell'_1-\alpha(2k+1),\ell_2=\ell'_2,\ldots,\ell_k=\ell'_k)$,
%then $\omega$ is a simple ordering of the half-set $L=\{\ell_1,\ell_2,\ldots,\ell_k\}$ of $\mathbb{Z}_{2k+1}$ which clearly is integer.

For any $j \in [0,p-1]$ define $\mathcal{P}_j=\{B_{i,j} \mid i \in[0, p-1]\}$, where 
\[
B_{i,j} = \{a_m(i+jm)(2k+1)+\ell_m \mid m\in[0,k-1] \}.
\] 
As a first remark, we note that each $\mathcal{P}_j$ is a simple Heffter system in $\Z_{p(2k+1)}$ relative to $\langle 2k+1\rangle$. Indeed,
\[
\begin{aligned}
\sum B_{i,j} &= \sum_{m=0}^{k-1} ( a_{m} (i+jm)(2k+1)+\ell_{m}) \\
&= (2k+1) \left(i \sum_{m=0}^{k-1} a_m+j\sum_{m=0}^{k-1}  ma_m \right)  + \sum L = 0,
\end{aligned}
\]
where the last equality holds by Lemma \ref{sequence lemma} if $k$ is odd, and by Lemma \ref{lem:seq_k_even} if $k$ is even. The simplicity of the blocks of each $\mathcal{P}_j$ and the fact that they partition a half-set of $\Z_{p(2k+1)} \setminus \langle 2k+1\rangle$ follows by  an argument analogous to the one of Proposition \ref{Heffter system 1}.

To verify that $\{\mathcal{P}_j \mid j \in [0,p-1]\}$ is a set of mutually orthogonal Heffter systems, assume by way of contradiction that there exist two distinct blocks $B_{i_1,j_1}$ and $B_{i_2,j_2}$ having two common elements $x$ and $y$. 
Then let $ m_1\in[0,k-1]$ be the index such that $x \equiv \ell_{m_1} \pmod{2k+1}$. Since $x \in B_{i_1,j_1}\cap B_{i_2,j_2}$, the following equation holds:
\[
a_{m_1} (i_1+j_1m_1) (2k+1) + \ell_{m_1} = a_{m_1} (i_2+j_2m_1) (2k+1) + \ell_{m_1} 
\]
That implies:
\begin{equation} \label{eq:all_k1}
i_1+j_1m_1 = i_2+j_2m_1.
\end{equation}
Similarly, if $y \equiv \ell_{m_2} \pmod{2k+1}$, then $y \in B_{i_1,j_1}\cap B_{i_2,j_2}$ implies:
\begin{equation} \label{eq:all_k2}
i_1+j_1m_2 = i_2+j_2m_2.
\end{equation}
If we subtract Equation (\ref{eq:all_k2}) from Equation (\ref{eq:all_k1}) we obtain the following:
\[
j_1(m_1-m_2) = j_2(m_1-m_2),
\]
that implies $j_1 = j_2$, hence $i_1=i_2$. This is a contradiction, hence $\{\mathcal{P}_j \mid j \in [0,p-1]\}$ is a simple cyclic $(pk,k;p)_{p}$ relative Heffter space.
\end{proof}

\begin{example}\label{ex:n7k6}
Take $p=7$ and $k = 6$, then $\mathcal{P}_0, \dotsc, \mathcal{P}_6$ listed below are the parallel classes of a simple cyclic $(42,6;7)_{7}$ relative Heffter space:

\begin{center}
\begin{tabular}{|c|c|c|c|}\hline
& $\mathcal{P}_0$& $\mathcal{P}_1$& $\mathcal{P}_2$ \\ \hline
   $B_{0,j}$&$\{  7 , 1 , -2 , 3 , -4 , -5\}$  &$\{  7 , 27 , 24 , -36 , -43 , 21\}$  &$\{  7 , -38 , -41 , 16 , 9 , -44\}$  \\ \hline 
   $B_{1,j}$&$\{  -19 , 27 , 11 , -10 , 9 , -18\}$  &$\{  -19 , -38 , 37 , 42 , -30 , 8\}$  &$\{  -19 , -12 , -28 , 3 , 22 , 34\}$  \\ \hline
  $B_{2,j}$&$\{  -45 , -38 , 24 , -23 , 22 , -31\}$  &$\{  -45 , -12 , -41 , 29 , -17 , -5\}$  &$\{  -45 , 14 , -15 , -10 , 35 , 21\}$  \\ \hline
  $B_{3,j}$&$\{  20 , -12 , 37 , -36 , 35 , -44\}$  &$\{  20 , 14 , -28 , 16 , -4 , -18\}$  &$\{  20 , 40 , -2 , -23 , -43 , 8\}$  \\ \hline
  $B_{4,j}$&$\{  -6 , 14 , -41 , 42 , -43 , 34\}$  &$\{  -6 , 40 , -15 , 3 , 9 , -31\}$  &$\{  -6 , -25 , 11 , -36 , -30 , -5\}$  \\ \hline
  $B_{5,j}$&$\{  -32 , 40 , -28 , 29 , -30 , 21\}$  &$\{  -32 , -25 , -2 , -10 , 22 , -44\}$  &$\{  -32 , 1 , 24 , 42 , -17 , -18\}$  \\ \hline
  $B_{6,j}$&$\{  33 , -25 , -15 , 16 , -17 , 8\}$  &$\{  33 , 1 , 11 , -23 , 35 , 34\}$  &$\{  33 , 27 , 37 , 29 , -4 , -31\}$  \\ \hline
 \end{tabular}
 \end{center}
 \begin{center}
 \begin{tabular}{|c|c|c|c|}\hline
& $\mathcal{P}_3$& $\mathcal{P}_4$& $\mathcal{P}_5$ \\ \hline
    $B_{0,j}$&$\{  7 , -12 , -15 , -23 , -30 , -18\}$  &$\{  7 , 14 , 11 , 29 , 22 , 8\}$  &$\{  7 , 40 , 37 , -10 , -17 , 34\}$  \\ \hline
    $B_{1,j}$&$\{  -19 , 14 , -2 , -36 , -17 , -31\}$  &$\{  -19 , 40 , 24 , 16 , 35 , -5\}$  &$\{  -19 , -25 , -41 , -23 , -4 , 21\}$  \\ \hline
    $B_{2,j}$&$\{  -45 , 40 , 11 , 42 , -4 , -44\}$  &$\{  -45 , -25 , 37 , 3 , -43 , -18\}$  &$\{  -45 , 1 , -28 , -36 , 9 , 8\}$  \\ \hline
    $B_{3,j}$&$\{  20 , -25 , 24 , 29 , 9 , 34\}$  &$\{  20 , 1 , -41 , -10 , -30 , -31\}$  &$\{  20 , 27 , -15 , 42 , 22 , -5\}$  \\ \hline
    $B_{4,j}$&$\{  -6 , 1 , 37 , 16 , 22 , 21\}$  &$\{  -6 , 27 , -28 , -23 , -17 , -44\}$  &$\{  -6 , -38 , -2 , 29 , 35 , -18\}$  \\ \hline
    $B_{5,j}$&$\{  -32 , 27 , -41 , 3 , 35 , 8\}$  &$\{  -32 , -38 , -15 , -36 , -4 , 34\}$  &$\{  -32 , -12 , 11 , 16 , -43 , -31\}$  \\ \hline
    $B_{6,j}$&$\{  33 , -38 , -28 , -10 , -43 , -5\}$  &$\{  33 , -12 , -2 , 42 , 9 , 21\}$  &$\{  33 , 14 , 24 , 3 , -30 , -44\}$  \\ \hline
  \end{tabular}

   \end{center}
 \begin{center}
  \begin{tabular}{|c|c|} \hline
  &$\mathcal{P}_6$ \\ \hline 
  $B_{0,j}$ &$\{  7 , -25 , -28 , 42 , 35 , -31\}$  \\ \hline
  $B_{1,j}$ &$\{  -19 , 1 , -15 , 29 , -43 , -44\}$  \\ \hline
 $B_{2,j}$ &$\{  -45 , 27 , -2 , 16 , -30 , 34\}$  \\ \hline
 $B_{3,j}$ &$\{  20 , -38 , 11 , 3 , -17 , 21\}$  \\ \hline
 $B_{4,j}$ &$\{  -6 , -12 , 24 , -10 , -4 , 8\}$  \\ \hline
 $B_{5,j}$ &$\{  -32 , 14 , 37 , -23 , 9 , -5\}$  \\ \hline
 $B_{6,j}$ &$\{  33 , 40 , -41 , -36 , 22 , -18\}$  \\ \hline
 \end{tabular}
\end{center}
\end{example}

In the last part of this section we focus on the special case in which the block size equals the degree of the space (and hence the number of points and blocks are equal).
Note that the following is a consequence both of Theorem \ref{thm:hs1} and Theorem \ref{thm:hs2}.
\begin{corollary}\label{cor:p2,p,p}
Given a prime $p\geq 3$, then there exists a simple cyclic $(p^2,p;p)_{p}$ relative Heffter space.
\end{corollary}
In the case of the above corollary we have a resolvable configuration where the number of points is the square of the block size, hence we find again the concept of a \emph{net}.
So, following the terminology introduced in \cite{BP}, we call the Heffter space of Corollary \ref{cor:p2,p,p} a Heffter net. The densest Heffter net known so far has been obtained, with the aid of a computer, in \cite{BP} and it has parameters $(121,11;9)$ and hence density $\delta=0.75$. The Heffter net constructed in Corollary \ref{cor:p2,p,p} has density $\delta=\frac{p}{p+1}$, which asymptotically approaches $1$.
We recall that in \cite{BP} (see Corollary 2.2) the authors proved that a linear cyclic Heffter space cannot exists, namely that it is not possible, working in a cyclic group, to construct a Heffter space with density one. Since the proof of this result also holds in the more general case where 
the point set is a half-set of $\mathbb{Z}_w\setminus J$ and $J$ is not necassarily the trivial subgroup, we can state the following.
\begin{proposition}
    A cyclic relative Heffter space cannot be linear.
\end{proposition}
This allows us state the make the following consideration.
\begin{remark} \label{rem:densest_HS}
For any prime $p\geq 3$, the $(p^2,p;p)_{p}$ Heffter net
of Corollary \ref{cor:p2,p,p} is the densest Heffter net which can be constructed
on $\mathbb{Z}_{p(2p+1)}$ relative to $\langle 2p+1 \rangle$.
\end{remark}

%\begin{remark}
%The Heffter spaces constructed in Corollary \ref{cor:np_p_p} have density $\frac{p(p-1)}{np-1}$, that is $\frac{p}{p+1}$ if $n=p$. As a first remark, the latter density asymptotically approaches $1$, proving the existence of cyclic non-linear Heffter spaces whose density is arbitrarily close to $1$. Moreover, a $(p^2,p;p+1)_{2p+1}$ Heffter space would be linear, but as done in Corollary 2.2 of \cite{BP} it can be shown that a cyclic and linear relative Heffter space cannot exist. Hence, the Heffter spaces constructed of Corollary \ref{cor:np_p_p} with $n=p$ are the densest that can be achieved with their corresponding set of parameters.
%\end{remark}

\section{Globally simple relative Heffter arrays}\label{sec:HA}
As explained in the Introduction, a relative Heffter space of degree two is completely equivalent to a relative Heffter array. This means that, as a consequence of the results obtained in the previous section, we get new constructions for relative Heffter arrays. Moreover, these arrays satisfy the very strong additional condition of being \emph{globally simple}, a property introduced in \cite{CMPP2}.

As usual, with a little abuse of notation, we can identify each row (respectively column) of a (relative) Heffter array $\H_t(n; k)$ with the set of size $k$ whose elements are the entries of the nonempty cells of such a row (respectively column). A (relative) Heffter array is \emph{simple} if each row and each column admits a simple ordering. Hence, to verify this property we need to provide an ordering for each row and each column which is simple. Clearly, larger
$n$ and $k$ are longer and more tedious is to provide explicit simple orderings for
rows and columns of an $\H_t(n;k)$. For this reason, in \cite{CMPP2} the authors introduced the concept of a \emph{globally simple} Heffter array, namely a Heffter array which is simple
with respect to the natural ordering of rows (namely from left to right) and columns (namely from top to bottom). It is evident that to construct globally simple (relative) Heffter arrays is much more difficult than to construct simple (relative) Heffter arrays.
Infinite classes of globally simple Heffter arrays can be found in \cite{BHeffter, BCDY, CDY, CMPP, DM, MT}. At the moment, as far as we know, there are only two classes of globally simple relative Heffter arrays, that is $\H_7(n;7)$ and $\H_9(n;9)$, constructed in \cite{CPP}. Hence in the following, we present the first two infinite classes of globally simple relative Heffter arrays in which the block size is not fixed.

\begin{theorem}\label{thm:array1}
There exists a globally simple $\H_n(n;k)$ for every odd $n\geq 3$ and every $k$ dividing $n$. 
\end{theorem}
\begin{proof}
Let $n=kd$ be odd. Let $\mathcal{H} = \{\mathcal{P}_s \mid  \text{ $s=0$ or $\gcd(s,k)=1$}\}$ be the $(nk,k;\Phi(k)+1)_n$ relative Heffter space of Theorem \ref{thm:hs1}, and consider the two Heffter systems $\mathcal{P}_0$ and $\mathcal{P}_{1}$. Construct the $n\times n$ array $H$ such that for every $g,h \in [0,d-1]$ and $i,j \in [0,k-1]$ the $(gk+i+1, hk+j+1)$-th cell of $H$ is filled with the element $B_{i,g,0}\cap B_{j,h,1}$ if it exists, and it is empty otherwise, where we recall that $B_{i,g,0} \in \mathcal{P}_0$ and $B_{j,h,1}\in\mathcal{P}_1$. Clearly, the array $H$ is an $\H_n(n;k)$; in what follows, we show that it is globally simple.

Assume that two blocks $B_{i,g,0}$ and $B_{j,h,1}$ share a common element $x$. We recall that for every $g,h \in [0,d-1]$  and $i,j\in [0,k-1]$:
 \begin{gather*}
    B_{i,g,0} = \{a_m(id+g)(2k+1)+\ell_m \mid m \in [0,k-1]\},
    \end{gather*} 
      \begin{gather*}
        B_{j,h,1} = \{a_{m+j}(md+h)(2k+1)+\ell_{m+j}\mid m \in[0,k-1]\}.
    \end{gather*}
    We have $x \in B_{i,g,0} \cap B_{j,h,1}$, with $x \equiv \ell_{m_1} \pmod{2k+1} \equiv \ell_{m_2+j} \pmod{2k+1}$ for some $m_1,m_2 \in [0,k-1]$, hence:     
    \[
    a_{m_1}(id+g)(2k+1)+\ell_{m_1} \equiv a_{m_2+j}(m_2d+h)(2k+1)+\ell_{m_2+j} \pmod{n(2k+1)}.
    \]
    Clearly, it follows that $m_1 \equiv m_2 + j \pmod{2k+1}$, thus $\ell_{m_1} = \ell_{m_2+j}$ and $a_{m_1}=a_{m_2+j}$. From the previous equation, we derive that $id+g \equiv m_2d+h \pmod{n}$; since $d$ divides $n$, it can easily be seen that necessarily $g=h$, hence $i=m_2$. As a first consequence, we have shown that $B_{i,g,0} \cap B_{j,h,1}$ is nonempty if and only if $g=h$. Moreover, as $i=m_2$, we immediately derive that the $(gk+j+1)$-th  column of $H$ is filled with the sequence:
    \[
     (a_{j}g(2k+1) + \ell_{j},a_{1+j}(d+g)(2k+1) + \ell_{1+j},  \ldots,a_{k-1+j}((k-1)d+g)(2k+1)+ \ell_{k-1+j}),
    \]
    that is simple by construction (see Proposition \ref{Heffter system 1}). Finally, from $m_1 \equiv m_2 + j \pmod{2k+1}$ and $i=m_2$ it can be seen that the $(gk+i+1)$-th row of $H$ is filled with the sequence:
     \[
     (a_{i}(id+j)(2k+1)+\ell_{i}, a_{i+1}(id+j)(2k+1)+\ell_{i+1}, \dotsc, a_{i+k-1}(id+j)(2k+1)+\ell_{i+k-1}),
     \]
     that is a cyclic permutation of the ordering
     \[
     (a_{0}(id+j)(2k+1)+\ell_{0},a_{1}(id+j)(2k+1)+\ell_{1},\dotsc, a_{k-1}(id+j)(2k+1)+\ell_{k-1} ),
     \]
     that is simple by construction (see again Proposition \ref{Heffter system 1}). Hence, the rows and columns of $H$ are simple with respect to their natural ordering, and the array $H$ is a globally simple $H_n(n;k)$.

\end{proof}

\begin{example}\label{ex:HA1}
The following is a globally simple $\H_{15}(15;5)$ whose rows (columns, respectively) correspond to the blocks of the Heffter system $\mathcal{P}_0$ ($\mathcal{P}_1$, respectively) constructed in Example \ref{ex:n15k5}. We recall that the entries of the array are elements in $\Z_{165}$.
\[
\begin{array}{|c|c|c|c|c|c|c|c|c|c|c|c|c|c|c|}\hline
-10 & -2 & 3 & 4 & 5 &   &   &   &   &   &   &   &   &   &  \\ \hline
-68 & 69 & -62 & 38 & 23 &   &   &   &   &   &   &   &   &   &  \\ \hline
-30 & 37 & 71 & 56 & 31 &   &   &   &   &   &   &   &   &   &  \\ \hline
-29 & -61 & -76 & -35 & 36 &   &   &   &   &   &   &   &   &   &  \\ \hline
-28 & -43 & 64 & -63 & 70 &   &   &   &   &   &   &   &   &   &  \\ \hline
  &   &   &   &   & 1 & -24 & 25 & -18 & 16 &   &   &   &   &  \\ \hline
  &   &   &   &   & 75 & -74 & 81 & 49 & 34 &   &   &   &   &  \\ \hline
  &   &   &   &   & -8 & 15 & 82 & 67 & 9 &   &   &   &   &  \\ \hline
  &   &   &   &   & -51 & -50 & -65 & -57 & 58 &   &   &   &   &  \\ \hline
  &   &   &   &   & -17 & -32 & 42 & -41 & 48 &   &   &   &   &  \\ \hline
  &   &   &   &   &   &   &   &   &   & 12 & -46 & 47 & -40 & 27\\ \hline
  &   &   &   &   &   &   &   &   &   & 53 & -52 & 59 & 60 & 45\\ \hline
  &   &   &   &   &   &   &   &   &   & 14 & -7 & -72 & 78 & -13\\ \hline
  &   &   &   &   &   &   &   &   &   & -73 & -39 & -54 & -79 & 80\\ \hline
  &   &   &   &   &   &   &   &   &   & -6 & -21 & 20 & -19 & 26\\ \hline
\end{array}
\]
\end{example}

The arrays constructed in Theorem \ref{thm:array1} have a block-diagonal structure, as shown in Example \ref{ex:HA1}, while the arrays we are going to construct in the next theorem have a diagonal structure, so it is convenient
to introduce the following notation.
If $H$ is an $n\times n$ array, for $i\in[1,n]$ we define the $i$-th diagonal
$$D_i=\{(i,1),(i+1,2),\ldots,(i-1,n)\}.$$
Here all arithmetic on the row and the column indices is performed modulo $n$, where the set of reduced residues is
$\{1,2,\ldots,n\}$.
We say that the diagonals $D_i,D_{i+1},\ldots, D_{i+r}$ are \emph{consecutive diagonals}.

\begin{definition}
  Let $k\geq 1$ be an integer. One says that a square Heffter array $H$ of size $n\geq k$ is \emph{cyclically $k$-diagonal} if the nonempty cells of $H$ are exactly those of $k$ consecutive diagonals.
%\begin{itemize}
% \item \emph{$k$-diagonal}   if the non empty cells of $H$ are exactly those of $k$ diagonals;
 %\item \emph{cyclically $k$-diagonal} if the nonempty cells of $H$ are exactly those of $k$ consecutive diagonals.
% \end{itemize}
\end{definition}

\begin{theorem}\label{thm:array2}
There exists a  cyclically $k$-diagonal globally simple $\H_p(p;k)$ for every prime $p\geq 3$ and every $k\in [3,p]$.  
\end{theorem}
\begin{proof}
 Let $p$ and $k$ be as in the statement. Let $\mathcal{H} = \{\mathcal{P}_s \mid s \in[0,p-1]\}$ be the $(pk,k;p)_p$ relative Heffter space of Theorem \ref{thm:hs2}, and consider the two Heffter systems $\mathcal{P}_0$ and $\mathcal{P}_{p-1}$.  Construct the $p \times p$ partially filled array $H$ whose $(i,j)$-th cell contains the element $B_{i-1,0} \cap B_{j-1,p-1}$ if it exists, and it is empty otherwise.  Clearly, $H$ is an $\H_p(p;k)$; in what follows, we show that it is cyclically $k$-diagonal and globally simple.

Assume that two blocks $B_{i-1,0} \in \mathcal{P}_0$ and $B_{j-1,p-1}\in \mathcal{P}_{p-1}$ share a common element $x$.
We recall that for each $i,j \in [0,p-1]$,
\[
\begin{aligned}
    B_{i-1,0} &= \{a_m(i-1)(2k+1)+ \ell_m \mid m\in[0,k-1]\} \\
    B_{j-1,p-1} &= \{a_m(j-1+m(p-1))(2k+1)+\ell_m \mid m\in[0,k-1]\},
\end{aligned}
\]
where $A = (a_0,\dotsc, a_{k-1})$ is the sequence of Lemma \ref{sequence lemma} if $k$ is odd, and of Lemma \ref{lem:seq_k_even} if $k$ is even, and $L = \{\ell_0, \dotsc, \ell_{k-1}\}$ is an integer half-set of $\Z_{2k+1}$ having a simple ordering $(\ell_0,\dotsc, \ell_{k-1})$, as shown in the proof of Theorem \ref{thm:hs2}.

From the expression of $x \in B_{i-1,0} \cap B_{j-1,n-1}$, with $x \equiv \ell_m \pmod{2k+1}$, we obtain the following equation:
\[
a_m (i-1)(2k+1)+ \ell_m \equiv a_m (j-1+m(p-1))(2k+1)+ \ell_m \pmod{p(2k+1)},
\]
that implies $i \equiv j+m(p-1) \pmod{p}$, hence $m \equiv j-i \pmod{p}$. Note that from this equivalence we deduce that the $(i,j)$-th cell of $H$ is filled if and only if $ j-i \pmod{p} \in [0, k-1]$, hence $H$ is cyclically $k$-diagonal. It then follows that the $i$-th row of $H$ read with respect to its natural ordering is a 
%ordered sequence of the elements contained in the $i$-th row of $A$ is a 
cyclic permutation of
\[
\omega = (a_0 (i-1) (2k+1)+\ell_0,a_1 (i-1)(2k+1) + \ell_1, \dotsc,
a_{k-2} (i-1)(2k+1)+\ell_{k-2},a_{k-1} (i-1)(2k+1)+\ell_{k-1}).
\] 
From Theorem \ref{thm:hs2} we have that $\omega$ is a simple ordering, and since a cyclic permutation of a simple ordering of a zero-sum set is simple as well,  the $i$-th row of $H$ read with respect to its natural ordering is simple. Similarly, it can be seen that the elements contained in the $j$-th column of $H$ read with respect to its natural ordering is a cyclic permutation of
\[
\begin{aligned}
\omega = (&a_{k-1} (j+(p-1)(k-1))(2k+1)+\ell_{k-1}, a_{k-2} (j+(p-1)(k-2))(2k+1)+\ell_{k-2}, \dotsc, \\ &a_{1} (j+p-1)(2k+1)+\ell_{1}, a_{0} j(2k+1)+\ell_{0}).
\end{aligned}
\] 
Therefore, also each column of $H$ is simple with respect to natural ordering.
We can conclude that $H$ is a cyclically $k$-diagonal globally simple  $\H_p(p;k)$. 
\end{proof}

\begin{remark}
In Theorems \ref{thm:array1} and \ref{thm:array2} we obtain the arrays starting from two particular parallel classes of the Heffter space constructed in Theorems \ref{thm:hs1} and \ref{thm:hs2}, respectively. Reasoning in the same way, it can be shown that a globally simple Heffter array can be constructed from each pair of distinct parallel classes. 
\end{remark}

\begin{example}
  The following is a globally simple cyclically $6$-diagonal $\H_7(7;6)$ whose rows (respectively columns) correspond to the blocks of the Heffter system $\mathcal{P}_0$
  (respectively $\mathcal{P}_6$) constructed in Example \ref{ex:n7k6}. We recall that the entries of the array are elements of $\mathbb{Z}_{91}$.
  
  %$$ \begin{array}{|r|r|r|r|r|r|r|} \hline
%-19 & 27 & 11 & -10 & 9 & -18 &  \\ \hline
%-38 & 24 & -23 & 22 & -31 &  &  -45\\ \hline
%37 & -36 & 35 & -44 &  & 20 & -12  \\ \hline
%42 & -43 & 34 & & -6 & 14 &  -41\\ \hline
%-30 & 21 &  & -32 & 40 & -28 &  29\\ \hline
%8 &  & 33 & -25 & -15 & 16 & -17 \\ \hline
% & 7 & 1 & -2 & 3 & -4 & -5 \\ \hline
%\end{array}
%$$

\[
\begin{array}{|r|r|r|r|r|r|r|} \hline
7 & 1 & -2 & 3 & -4 & -5 & \\ \hline
 & -19 & 27 & 11 & -10 & 9 & -18\\ \hline
-31 &  & -45 & -38 & 24 & -23 & 22\\ \hline
35 & -44 &  & 20 & -12 & 37 & -36\\ \hline
42 & -43 & 34 &  & -6 & 14 & -41\\ \hline
-28 & 29 & -30 & 21 &  & -32 & 40\\ \hline
-25 & -15 & 16 & -17 & 8 &  & 33\\ \hline
\end{array}
\]

\end{example}

\section{Orthogonal cycle decompositions and biembeddings}\label{sec:decomposition}
It is well known that Heffter arrays give rise to graph decompositions obtainable via difference
methods (see Section 5 of \cite{PD} for details). More generally, in \cite{BP} the authors use Heffter spaces to construct sets of mutually orthogonal cycle systems, and the same can also be done starting from \emph{relative} Heffter spaces. To explain this we firstly introduce some background on this topic.
By $K_{m\times n}$ we will denote the complete multipartite graph
with $m$ parts of size $n$, and we recall that a \textit{$k$-cycle decomposition}  of $K_{m\times n}$ is a set of $k$-cycles whose edges partition the edge-set of $K_{m\times n}$. Such a decomposition $\mathcal{D}$ is said to be $G$-\emph{regular} if the vertex set of $K_{m\times n}$ is an additive group $G$ and $C+g \in \mathcal{D}$ for every pair $(C,g)\in \mathcal{D} \times G$. Equivalently if, up to isomorphism, $G$ is an automorphism group of $\mathcal{D}$. If $G$ is a cyclic group one simply speaks of a \emph{cyclic}  decomposition.
Two cycle decompositions, say $\mathcal{D}$ and ${\mathcal{D}}'$, of a graph $K$ are \emph{orthogonal} if there is no cycle of $\mathcal{D}$ sharing more than one edge with a cycle of ${\mathcal{D}}'$. 

The construction of a set of mutually orthogonal cycle decompositions of the complete graph was recently considered in \cite{BP1,BP,BCP,KY}, but to our knowledge, there are no results on sets of size greater than two of mutually orthogonal cycle decompositions of the complete multipartite graph. 
On the other hand, if there exists a simple relative Heffter array, then there exist two orthogonal cycle decompositions of the complete multipartite graph. To explain this we have to introduce some notation. Given an $n \times n$ partially filled array $H$,
we will denote by $\E(H)$ the set of the elements of the filled cells of $H$.
Analogously, by $\E(R_i)$ and $\E(C_j)$ we mean the elements of the $i$-th row and of the $j$-th column, respectively, of $H$.
Also, by $\omega_{R_i}$ and $\omega_{C_j}$ we will
denote, respectively,  an ordering of $\E(R_i)$ and of $\E(C_j)$.  If for any $i,j\in[1,n]$, the orderings $\omega_{R_i}$ and $\omega_{C_j}$ are simple, we denote by
  $\omega_r=\omega_{R_1}\circ \ldots \circ\omega_{R_n}$ the simple ordering for the rows and
  by $\omega_c=\omega_{C_1}\circ \ldots \circ\omega_{C_n}$ the simple ordering for the columns.
 The relationship between simple relative Heffter arrays and cyclic cycle decompositions of the complete multipartite graph is explained in detail in
\cite{CMPP}. Here we briefly recall the following result.
 \begin{proposition}\label{HeffterToDecompositions}\cite[Proposition 2.9]{CMPP}
  Let $H$ be a simple $\H_t(n;k)$ with respect to the orderings $\omega_r$ and $\omega_c$. Then there exist two cyclic $k$-cycle decompositions $\D_{\omega_r}$ and $\D_{\omega_c}$ of $K_{\frac{2nk+t}{t}\times t}$. Moreover the decompositions $\D_{\omega_r}$ and $\D_{\omega_c}$ are orthogonal.
\end{proposition}

Since the relative Heffter spaces constructed in Section \ref{sec:main} are simple,
here we obtain, as a consequence, sets with many mutually orthogonal cycle decompositions of the complete multipartite graph. In fact, it is not hard to see that the following proposition holds.
\begin{proposition}\label{prop:space+decomp}
 If there exists a simple $(nk, k;r)_t$ relative Heffter space over $G$, then there exist $r$ mutually orthogonal $G$-regular $k$-cycle decompositions of $K_{\frac{nk}{t}\times t}$.
\end{proposition}
The above proposition is nothing but a generalization of Proposition \ref{HeffterToDecompositions} and the proof can be obtained reasoning in the same way.
This connection between simple relative Heffter spaces and cycle decompositions allows us to state the following results.

\begin{theorem}\label{Heffter system existence prop}
  Let $n$ be an odd integer and $k$ be a divisor of $n$. Then there exist at least $\Phi(k)+1$ mutually orthogonal
  cyclic $k$-cycle decompositions of $K_{(2k+1)\times n}$.
\end{theorem}
\begin{proof}
The result follows by Theorem \ref{thm:hs1} and Proposition \ref{prop:space+decomp}.
\end{proof}

\begin{theorem}
    For every prime $p\geq 3$ and $k\in[3,p]$, there exist at least $p$ mutually orthogonal cyclic $k$-cycle decompositions of $K_{(2k+1)\times p}$.
\end{theorem}
\begin{proof}
The result follows by Theorem \ref{thm:hs2} and Proposition \ref{prop:space+decomp}.
\end{proof}
Since the decompositions so constructed are cyclic, to describe them it is sufficient to give a set of the so- called \emph{base blocks}, then all the
other cycles of the decompositions can be obtained considering the orbit of the base blocks under the natural action of the cyclic group. 
\begin{example}
Starting from the blocks of the seven parallel classes of the simple relative Heffter space presented in Example \ref{ex:n7k6}, we construct the base blocks of seven mutually orthogonal cyclic  $6$-cycle decompositions of $K_{13 \times 7}$ under the action of the group $\mathbb{Z}_{91}$. In the following, let $\mathcal{D}_j$ denote the decomposition obtained from the parallel class $\mathcal{P}_j$, where $j\in[0,6]$. In detail, let $C_{i,j}$ be the graph whose vertices are the ordered partial sums of $B_{i,j}$, for each $i,j\in [0,6]$. Since each block $B_{i,j}$ has size $6$, sums to zero, and the ordering is simple, the resulting graph $C_{i,j}$ is a $6$-cycle. In particular we denote the vertices of the cycles with an element in the range $[-45,45]$, since, in this way, it is easier to check that the vertices are actually pairwise distinct.

\begin{center}
\begin{tabular}{|c|c|c|c|}\hline
& $\mathcal{D}_0$& $\mathcal{D}_1$& $\mathcal{D}_2$ \\ \hline
   $C_{0,j}$&$(  7 , 8 , 6 , 9 , 5 , 0)$  &$(  7 , 34 , -33 , 22 , -21 , 0)$  &$(  7 , -31 , 19 , 35 , 44 , 0)$  \\ \hline 
   $C_{1,j}$&$(  -19 , 8 , 19 , 9 , 18 , 0)$  &$(  -19 , 34 , -20 , 22 , -8 , 0)$  &$(  -19 , -31 , 32 , 35 , -34 , 0)$  \\ \hline
  $C_{2,j}$&$(  -45 , 8 , 32 , 9 , 31 , 0)$  &$(  -45 , 34 , -7 , 22 , 5 , 0)$  &$(  -45 , -31 , 45 , 35 , -21 , 0)$  \\ \hline
  $C_{3,j}$&$(  20 , 8 , 45 , 9 , 44 , 0)$  &$( 20 , 34 , 6 , 22 , 18 , 0)$  &$( 20 , -31 , -33 , 35 , -8 , 0)$  \\ \hline
  $C_{4,j}$&$(  -6 , 8 , -33 , 9 , -34 , 0)$  &$( -6 , 34 , 19 , 22 , 31 , 0)$  &$(  -6 , -31 , -20 , 35 , 5 , 0)$  \\ \hline
  $C_{5,j}$&$( -32 , 8 , -20 , 9 , -21 , 0)$  &$(-32 , 34 , 32 , 22 , 44 , 0)$  &$(-32 , -31 , -7 , 35 , 18 , 0)$  \\ \hline
  $C_{6,j}$&$( 33 , 8 , -7 , 9 , -8 , 0)$  &$( 33 , 34 , 45 , 22 , -34 , 0)$  &$( 33 , -31 , 6 , 35 ,  31 , 0)$  \\ \hline
 \end{tabular}
 \end{center}
 \begin{center}
 \begin{tabular}{|c|c|c|c|}\hline
& $\mathcal{D}_3$& $\mathcal{D}_4$& $\mathcal{D}_5$ \\ \hline
    $C_{0,j}$&$(  7 , -5 , -20 , -43 , 18 , 0)$  &$(  7 , 21 , 32 , -30 , -8 , 0)$  &$(  7 , -44 , -7 , -17 , -34 , 0)$  \\ \hline
    $C_{1,j}$&$( -19 , -5 , -7 , -43 , 31 , 0)$  &$(  -19 , 21 , 45 , -30 , 5 , 0)$  &$(  -19 , -44 , 6 , -17 , -21 , 0)$  \\ \hline
    $C_{2,j}$&$(  -45 , -5 , 6 , -43 , 44 , 0)$  &$(  -45 , 21 , -33 , -30 ,18 , 0)$  &$(  -45 , -44 , 19 , -17 , -8 , 0)$  \\ \hline
    $C_{3,j}$&$(  20 , -5 , 19 , -43 , -34 , 0)$  &$(  20 , 21 , -20 , -30 , 31 , 0)$  &$(  20 , -44 , 32 , -17 , 5 , 0)$  \\ \hline
    $C_{4,j}$&$(  -6 , -5 , 32 , -43 , -21 , 0)$  &$(  -6 , 21 , -7 , -30 , 44 , 0)$  &$( -6 , -44 , -45 , -17 , 18 , 0)$  \\ \hline
    $C_{5,j}$&$( -32 , -5 , 45 , -43 , -8 , 0)$  &$( -32 , 21 , 6 , -30 , -34 , 0)$  &$( -32 , -44 , -33 , -17 , 31 ,0)$  \\ \hline
    $C_{6,j}$&$(  33 , -5 , -33 , -43 , 5 , 0)$  &$( 33 , 21 , 19 , -30 , -21 , 0)$  &$(  33 , -44 , -20 , -17 , 44 , 0)$  \\ \hline
  \end{tabular}

   \end{center}
 \begin{center}
  \begin{tabular}{|c|c|} \hline
  &$\mathcal{D}_6$ \\ \hline 
  $C_{0,j}$ &$(  7 , -18 , -46 , -4 , 31 , 0)$  \\ \hline
  $C_{1,j}$ &$(  -19 , -18 , -33 , -4 , -44 , 0)$  \\ \hline
 $C_{2,j}$ &$(  -45 , -18 , -20 , -4 , -34 , 0)$  \\ \hline
 $C_{3,j}$ &$(  20 , -18 , -7 , -4 , -21 , 0)$  \\ \hline
 $C_{4,j}$ &$(  -6 , -18 , 6 , -4 , -8 , 0)$  \\ \hline
 $C_{5,j}$ &$(  -32 , 18 , 19 , -4 , 5 , 0)$  \\ \hline
 $C_{6,j}$ &$(  33 , -18 , 32 , -4 , 18 , 0)$  \\ \hline
 \end{tabular}
\end{center}
\end{example}

We conclude this section by showing that, as a consequence of results in Section 4, we can obtain new results concerning biembeddings of cycle decompositions. Actually, in \cite{A}, one Archdeacon's main motivations for defining Heffter arrays was due to their applications, in particular, owing to their usefulness in identifying biembeddings of cycle decompositions.
%Actually, in \cite{A}, Archdeacon introduced Heffter arrays also in view of their applications and, in particular, since they are useful for finding biembeddings of cycle decompositions. 
Then in \cite{CPP}, generalizing some of Archdeacon's results, the authors showed how starting from a relative
Heffter array it is also possible to obtain suitable biembeddings. Firstly, we need to recall some definitions and results,
we start from the following definition, see \cite{Moh}.

\begin{definition}
An \emph{embedding} of a graph $\Gamma$ in a surface $\Sigma$ is a continuous injective
mapping $\psi: \Gamma \to \Sigma$, where $\G$ is viewed with the usual topology as $1$-dimensional simplicial complex.
\end{definition}

The  connected components of $\Sigma \setminus \psi(\Gamma)$ are called $\psi$-\emph{faces}.
If each $\psi$-face is homeomorphic to an open disc, then the embedding $\psi$ is said to be \emph{cellular}.

\begin{definition}
 A \emph{biembedding} of  two cycle decompositions $\D$ and $\D'$ of a simple graph $\Gamma$  is a face $2$-colorable embedding
  of $\G$ in which one color class is comprised of the cycles in $\D$ and
  the other class contains the cycles in $\D'$.
\end{definition}

Given a  relative Heffter array $H=\H_t(n;k)$, the orderings $\omega_r$ and $\omega_c$ are said to be
\emph{compatible} if
$\omega_c \circ \omega_r$ is a cycle of length  $\E(H)$. The connection between relative Heffter arrays and biembeddings has been established in \cite{CPP} with the following result.

\begin{theorem}\cite[Theorem 3.4]{CPP}\label{thm:biembedding}
Let $H$ be a  relative Heffter array $\H_t(n;k)$ that is simple with respect to the compatible orderings $\omega_r$
and $\omega_c$.
Then there exists a cellular biembedding of the cyclic $k$-cycle decompositions $\mathcal{D}_{\omega_r^{-1}}$ and
$\mathcal{D}_{\omega_c}$ of $K_{\frac{2nk+t}{t}\times t}$ into an orientable surface of genus
$$g=1+\frac{(nk-2n-1)(2nk+t)}{2}.$$
\end{theorem}

The arrays constructed in the previous section are not only simple, but they are globally simple. Looking for compatible orderings in the case of a globally simple Heffter array led to investigate the following problem
introduced in \cite{CDP}.
Let $A$ be an $m\times n$ toroidal partially filled array. By $r_i$ we denote the orientation of the $i$-th row,
precisely $r_i=1$ if it is from left to right and $r_i=-1$ if it is from right to left. Analogously, for the $j$-th
column, if its orientation $c_j$ is  from  top to bottom then $c_j=1$ otherwise $c_j=-1$. Assume that an orientation
$\R=(r_1,\dots,r_m)$
 and $\C=(c_1,\dots,c_n)$ is fixed. Given an initial filled cell $(i_1,j_1)$ consider the sequence
$ L_{\R,\C}(i_1,j_1)=((i_1,j_1),(i_2,j_2),\ldots,(i_\ell,j_\ell),$ $(i_{\ell+1},j_{\ell+1}),\ldots)$
where $j_{\ell+1}$ is the column index of the filled cell $(i_\ell,j_{\ell+1})$ of the row $R_{i_\ell}$ next
to
$(i_\ell,j_\ell)$ in the orientation $r_{i_\ell}$,
and where $i_{\ell+1}$ is the row index of the filled cell of the column $C_{j_{\ell+1}}$ next to
$(i_\ell,j_{\ell+1})$ in the orientation $c_{j_{\ell+1}}$.
The problem is the following:

\begin{KN}
Given a toroidal partially filled array $H$,
do there exist $\R$ and $\C$ such that the list $L_{\R,\C}$ covers all the filled
cells of $H$?
\end{KN}

By $P(H)$ we will denote the \probname\ for a given array $H$.
Also, given a filled cell $(i,j)$, if $L_{\R,\C}(i,j)$ covers all the filled positions of $H$ we will
say that  $(\R,\C)$ is a solution of $P(H)$.
The relationship between the Crazy Knight's Tour Problem and globally simple relative Heffter arrays is explained in the following result which is an easy consequence of Theorem \ref{thm:biembedding}.

\begin{corollary}\cite[Corollary 3.5]{CPP}\label{preprecedente}
  Let $H$ be a globally simple relative Heffter array $\H_t(n;k)$  such that $P(H)$ admits a solution $(\R,\C)$.
  Then there exists a biembedding of the cyclic cycle decompositions $\mathcal{D}_{\omega_r^{-1}}$ and
$\mathcal{D}_{\omega_c}$ of $K_{\frac{2nk+t}{t}\times t}$ into an orientable surface.
\end{corollary}

In \cite{CDP} the authors proved that $P(H)$ admits a solution for several classes of (not necessarily square) arrays, here we recall only the results we need for the purpose of this paper.

\begin{theorem}\cite[Theorem 3.3]{CDP}
Let $H$ be a totally filled square array of order $n$. Then there exists a solution of
$P(H)$ if and only if $n$ is odd.
\end{theorem}

\begin{proposition}\cite[Proposition 4.6]{CDP}\label{prop:PH}
 Let $H$ be a cyclically $k$-diagonal array of size $n > k$. If $P(H)$ admits a solution
then $n$ and $k$ are both odd and $k\geq3$.
\end{proposition}
The necessary conditions of Proposition \ref{prop:PH} are also sufficient if $k\in[3,200]$, see \cite[Theorem 4.11]{CDP}.
\begin{proposition}\cite[Proposition 3.4]{CMPP2}\label{prop:PH2}
Let $k\geq3$ be an odd integer and let $H$ be a cyclically $k$-diagonal array of size
$n>k$. If $\gcd(n; k-1) = 1$, then $P(H)$ admits a solution.
\end{proposition}

Our main result about biembeddings is the following.
\begin{theorem}
For every prime $p\geq 3$ and every $k\in[3,p]$, there exists a biembedding
of the cyclic $k$-cycle decompositions $\mathcal{D}_{\omega_r^{-1}}$ and
$\mathcal{D}_{\omega_c}$ of $K_{(2k+1)\times p}$ into an orientable surface.
\end{theorem}
\begin{proof}
The result follows from Theorem \ref{thm:array2}, Corollary \ref{preprecedente} and Proposition \ref{prop:PH2}.
\end{proof}

\section*{Acknowledgements}
The research of the first author was funded by the LMS Cecil King Travel Scholarship and later supported by the Additional Funding Programme for Mathematical Sciences, delivered by EPSRC (EP/V521917/1) and the Heilbronn Institute for Mathematical Research. 
The second and the third author are partially supported by INdAM - GNSAGA. 
This work was supported in part by the European Union under the Italian National Recovery and Resilience Plan (NRRP) of NextGenerationEU, Partnership on “Telecommunications of the Future,” Program “RESTART” under Grant PE00000001, “Netwin” Project (CUP E83C22004640001).

\end{document}